\documentclass[11pt]{article}
\usepackage{amsmath,amssymb,latexsym, times}
\usepackage{amsthm}
\usepackage{rsfso}

\usepackage{amscd}
\usepackage{booktabs} % Required for nicer horizontal rules in tables
\usepackage{wrapfig}
\usepackage{tikz}
\usepackage{pgfplots}
\usepackage{float}
\pgfplotsset{compat=newest}
\usepackage{authblk}
\usepackage{mathrsfs}
\usepackage{bbm}
\usepackage{graphicx}
\usepackage{stmaryrd}
\usepackage[all]{xy}
\usepackage{cases}
\usepackage{enumitem}
\usepackage{tikz-cd}

\usepackage{tocloft}
\usepackage{hyperref}
\hypersetup{
    colorlinks=true,
    linkcolor=blue,
    filecolor=magenta,      
    urlcolor=cyan,
    citecolor=magenta
}

\usepackage{setspace}
\newtheorem{innercustomgeneric}{\customgenericname}
\providecommand{\customgenericname}{}
\newcommand{\newcustomtheorem}[2]{%
  \newenvironment{#1}[1]
  {%
   \renewcommand\customgenericname{#2}%
   \renewcommand\theinnercustomgeneric{##1}%
   \innercustomgeneric
  }
  {\endinnercustomgeneric}
}

\newcustomtheorem{customthm}{Theorem}
\newcustomtheorem{customconj}{Conjecture}

\theoremstyle{plain}
\newtheorem{thm}{Theorem}[section]
\newtheorem{lem}[thm]{Lemma}
\newtheorem{prop}[thm]{Proposition}
\newtheorem{cor}[thm]{Corollary}

\newtheorem{thm-def}[thm]{Theorem-Definition}

\theoremstyle{definition}
\newtheorem{defn}[thm]{Definition}
\newtheorem{cor-def}[thm]{Corollary-Definition}
\newtheorem{example}[thm]{Example}

\theoremstyle{remark}
\newtheorem{rem}[thm]{Remark}

\numberwithin{equation}{section}

\makeatletter
\renewcommand{\paragraph}{%
  \@startsection{paragraph}{4}%
  {\z@}{1.25ex \@plus 1ex \@minus .2ex}{-1em}%
  {\normalfont\normalsize\bfseries}%
}
\makeatother

\newcommand{\gl}{\mathrm{GL}}
\newcommand{\ggl}{\mathfrak{gl}}

\newcommand{\cf}{\mathcal{F}}
\newcommand{\co}{\mathcal{O}}
\newcommand{\cp}{\mathcal{P}}

\newcommand{\cn}{\mathcal{N}}

\newcommand{\ch}{\mathcal{H}}

\newcommand{\ce}{\mathcal{E}}

\newcommand{\xx}{\mathcal{X}}

\newcommand{\ka}{\mathfrak{a}}
\newcommand{\kg}{\mathfrak{g}}

\newcommand{\km}{\mathfrak{m}}

\newcommand{\bn}{\mathbf{N}}
\newcommand{\bz}{\mathbf{Z}}
\newcommand{\br}{\mathbf{R}}
\newcommand{\bp}{\mathbf{P}}
\newcommand{\bq}{\mathbf{Q}}

\newcommand{\bc}{\mathbf{C}}

\newcommand{\ba}{\mathbf{A}}

\newcommand{\rd}{\mathrm{d}}

\newcommand{\rtt}{\mathrm{T}}

\newcommand{\rS}{\mathrm{S}}

\newcommand{\val}{\mathrm{val}}

\newcommand{\Ad}{\mathrm{Ad}}

\newcommand{\codim}{\mathrm{Codim}}
\newcommand{\diag}{\mathrm{diag}}
\newcommand{\gal}{\mathrm{Gal}}
\newcommand{\Id}{\mathrm{Id}}

\newcommand{\reg}{\mathrm{reg}}
\newcommand{\spec}{\mathrm{Spec}}

\newcommand{\pic}{\mathrm{Pic}}

\newcommand{\can}{\mathrm{can}}
\newcommand{\var}{\mathrm{var}}

\newcommand{\spf}{\mathrm{Spf}}

\newcommand{\pr}{\mathrm{pr}}
\newcommand{\Hom}{\mathrm{Hom}}

\newcommand{\gr}{\mathrm{Gr}}

\newcommand{\rk}{\mathrm{rk}}

\newcommand{\ep}{\epsilon}
\newcommand{\varep}{\varepsilon}

\newcommand{\sm}{\mathrm{sm}}

\newcommand{\defm}{\mathrm{Def}}

\newcommand{\jac}{\mathrm{Jac}}
\newcommand{\ql}{\mathbf{Q}_{\ell}}

\newcommand{\cm}{{\mathcal{M}}}
\newcommand{\ct}{\mathcal{T}}

\newcommand{\cc}{\mathcal{C}}

\def\cb{{\mathcal B}}

\def\ka{{\mathfrak{a}}}

\def\kg{{\mathfrak{g}}}

\def\km{{\mathfrak{m}}}

\def\kbb{{\mathfrak{B}}}
\def\kaa{{\mathfrak{A}}}

\usepackage{xcolor}
\usepackage{version}

\excludeversion{NB}

\newcommand{\pair}[1]{\langle #1\rangle}

\newcommand{\bsm}{\begin{smallmatrix}}
\newcommand{\esm}{\end{smallmatrix}}

\def\cqfd{\hfill $\Box$}

\def\<{\langle\,}
\def\>{\,\rangle}

\def\kg{\mathfrak g}

\def\<{\langle}
\def\>{\rangle}
\def\={\equiv}
\def\le{\leqslant}
\def\ge{\geqslant}

\def\Hom{{\rm Hom}}

\def\rd{{\rm d}}
\def\val{{\rm val}}

\def\gal{{\rm Gal}}

\def\bl{{\rm Bl}}

\def\Id{{\rm Id}}

\author{Zongbin \textsc{Chen}}

\title{\textbf{On the dependence of the affine Springer fibers on their root valuation datum}}
\begin{document}
\maketitle

\begin{abstract}

For the group $\mathrm{GL}_{d}$, we confirm a conjecture of Goresky, Kottwitz and MacPherson, which states that the cohomology of the affine Springer fibers depend only on the root valuation datum of their defining elements.
The proof relies on a microlocal analysis of the intermediate extensions appearing in Ng\^o's support theorem, and a Whitney regularity property of the union of the strict $\delta$-strata over the equisingular strata for the miniversal deformation of plane curve singularities. 

\end{abstract}

\tableofcontents

\section{Introduction}

Let $F=\bc(\!(\varep)\!)$ be the field of Laurent series over $\bc$, $\co=\bc[\![\varep]\!]$ its ring of integers.
We fix an algebraic closure $\overline{F}=\bigcup_{n\in \bn}\bc(\!(\varep^{1/n})\!)$ of $F$, let $\tau\in \gal(\overline{F}/F)$ be the unique element that sends $\varep^{1/n}$ to $\zeta_{n}\varep^{1/n}$ for any $n\in \bn$, where $\{\zeta_{n}\}_{n\in \bn}$ is a compatible family of primitive $n$-th root of unity. Let $\val:\overline{F}\to \bq$ be the valuation normalized by the condition $\val(\varep)=1$.

Let $G$ be the general linear group $\gl_{d}$ over $\bc$, and let $\kg$ be its Lie algebra. 
Let $\gamma\in \kg[\![\varep]\!]$ be a regular semisimple element. The \emph{affine Springer fiber} at $\gamma$ is the closed sub ind-$\bc$-scheme of the \emph{affine grassmannian} $\xx=G(\!(\varep)\!)/G[\![\varep]\!]$ defined by
$$
\xx_{\gamma}=\Big\{g\in \xx\,\big|\,\Ad(g^{-1})\gamma \in \kg[\![\varep]\!]\Big\}.
$$ 
It is a scheme locally of finite type over $\bc$.
Their analogues over finite fields have been used to study the orbital integrals \cite{gkm homo}. 
Let $A$ be the maximal torus of $G$ consisting of the diagonal matrices, let $\Phi$ be the root system of $G$ with respect to $A$, let $W$ be the associated Weyl group. Goresky, Kottwitz and MacPherson \cite{gkm root} have defined the \emph{root valuation datum} of $\gamma$: Choose an element $\gamma'\in \ka(\overline{F})$ conjugate to $\gamma$, we can attach to $\gamma'$ a pair $(w,r)$, where $w\in W$ is the unique element such that $w\tau(\gamma')=\gamma'$, and $r:\Phi\to \bq$ is the function sending $\alpha\in \Phi$ to $\val(\alpha(\gamma'))$. The element $\gamma'$ is not unique, we can replace it with any element of the form $w'\gamma'$, $w'\in W$, for which the pair $(w,r)$ is replaced by $(w'w{w'}^{-1}, w'r)$, where $w'r(\alpha)=r({w'}^{-1}(\alpha))$.   
The equivalence class of the pair $(w,r)$ under the above action is called the \emph{root valuation datum} of $\gamma$.

\begin{customthm}{1}[Conjecture of Goresky-Kottwitz-MacPherson, \cite{gkm root}]\label{main}

The cohomology of the affine Springer fiber $\xx_{\gamma}$ depends only on the root valuation datum of the  element $\gamma$. 

\end{customthm}

The theorem is proven via deformation of the compactified Jacobians. 
Recall that $\xx_{\gamma}\cong \xx_{\gamma}^{0}\times \bz$ with $\xx_{\gamma}^{0}$ being  the central connected component of $\xx_{\gamma}$, and $\xx_{\gamma}^{0}$ admits the action of a discrete free abelian group $\Lambda^{0}$. 
According to Laumon \cite{laumon springer}, the quotient $\Lambda^{0}\backslash \xx_{\gamma}^{0}$ is homeomorphic to the compactified Jacobian $\overline{P}_{C_{\gamma}}$ of a spectral curve $C_{\gamma}$ over $\bc$, and $\xx_{\gamma}^{0}$ is homeomorphic to a $\Lambda^{0}$-torsor $\overline{P}{}_{C_{\gamma}}^{\natural}$ over $\overline{P}_{C_{\gamma}}$, which can be seen as a universal abelian covering of $\overline{P}_{C_{\gamma}}$.
Moreover, the construction of $\overline{P}{}_{C_{\gamma}}^{\natural}$ can be put in family, and theorem \ref{main} can be reduced to:

\begin{customthm}{2}\label{main CJ}

The cohomology of $\overline{P}_{C_{\gamma}}$ depends only on the root valuation datum of the element $\gamma$. 

\end{customthm}

In terms of plane curve singularities, the root valuation datum of $\gamma$ determines the datum consisting of the Puiseux characteristics of the irreducible components (called \emph{branches})  of the singularity $\spf(\co[\gamma])$ and the intersection numbers between pairs of the branches. According to Zariski \cite{zariski equisingularity}, plane curve singularities with the same such datum will be \emph{equisingular}. Hence it is enough to show that the cohomology of $\overline{P}_{C_{\gamma}}$ is locally constant under equisingular deformation of $C_{\gamma}$.

Let $\pi:\cc\to \cb$ be an algebraization of a miniversal deformation of $C_{\gamma}$, with $C_{\gamma}$ lying over $0\in \cb$. 
It is known that $\cb$ is smooth over $\bc$ and we can even take $\cb=\ba^{\tau_{\gamma}}$, where $\tau_{\gamma}$ is the Tjurina number of the singularity $\spf(\co[\gamma])$.
Then $\overline{P}_{C_{\gamma}}$ fits in the projective flat family of relative compactified Jacobians 
$$
f:\overline{\cp}=\overline{\pic}{}^{\,0}_{\cc/ \cb}\to \cb.
$$ 
Let $j:\cb^{\rm sm}\to \cb$ be the open subscheme over which $f$ is smooth, and let $f^{\rm sm}$ be the restriction of $f$ to the inverse image of $\cb^{\rm sm}$. Ng\^o's support theorem \cite{ngo} states that
\begin{equation*}
Rf_{*}\ql=\bigoplus_{i=0}^{2\delta_{\gamma}} j_{!*}(R^{\,i}f^{\sm}_{*}\ql)[-i],
\end{equation*}
where $\delta_{\gamma}$ is the $\delta$-invariant of $C_{\gamma}$.
In particular,
\begin{equation*}
H^{i}\big(\overline{P}_{C_{\gamma}}, \ql\big)=\bigoplus_{i'=0}^{i}\ch^{i-i'}\big(j_{!*}R^{i'}f^{\sm}_{*}\ql\big)_{0},\quad i=0,\cdots,2\delta_{\gamma}.
\end{equation*}
Consider the partition of the base $\cb$ as union of locally closed subschemes such that fibers of the family $\pi:\cc\to \cb$ are equisingular over each subscheme, we call it the \emph{equisingular stratification}, and call the resulting strata \emph{equisingular strata}\footnote{Here the word \emph{stratification} is used in an abusive manner, it seems unknown whether the equisingular stratification satisfies the frontier condition for a stratification.}.  
To prove theorem \ref{main CJ}, we need to show that the sheaves $\ch^{*}(j_{!*}R^{\,i}f^{\sm}_{*}\ql)$ are locally constant over the equisingular strata for all $i$. 
If the equisingular stratification was a Whitney stratification, theorem \ref{main CJ} will be an immediate consequence of Ng\^o's support theorem, as the intermediate extension is locally constant along the Whitney strata. 
But this is too much to ask, there is a counter-example of Fr\'ed\'erique Pham (cf. \cite{bert pham}), which shows that the equisingular stratification is not fine enough to be a Whitney stratification.

Here some miracle happens. The local regularity criteria of Fantechi-G\"ottsche-van Straten \cite{fgv} about the family $f:\overline{\cp}\to \cb$ implies a local acyclicity property of the complex $j_{!*}R^{\,i}f^{\sm}_{*}\ql$ (theorem \ref{local acyclicity}), to the effect that it is enough to show that the union of the \emph{strict $\delta$-strata} are Whitney regular over the equisingular strata which lie at their frontier. Here the \emph{strict $\delta$-strata} refer to the equisingular strata over which the curves have $\delta$ ordinary double points as singularities.   
With Teissier's criteria \cite{teissier polar 2}, the Whitney regularity property in question follows from the smoothness of the equisingular strata and a wonderful property of the discriminant $\Delta$ of the family $\pi$, that its Nash blow-up is smooth \cite{teissier hunting}.

\subsection*{Notations, conventions and useful facts}

\paragraph{(1)}

Let $X$ be a locally noetherian scheme, we denote by $X^{\reg}$ its regular locus and $X^{\rm sing}=X-X^{\reg}$ the singular locus.

\paragraph{(2)}

Let $f:X\to S$ be a morphism of schemes, we will denote by $\Delta_{f}$ the discriminant locus of $f$, and by $S^{\sm}$ the open subscheme $S-\Delta_{f}$ over which $f$ is smooth, and $f^{\sm}:X^{\rm sm}\to S^{\sm}$  the  restriction of $f$ to the inverse image of $S^{\sm}$. For any point $s\in S$, we denote by $X_{s}$ the fiber of $f$ at $s$.

\paragraph{(3)} 

Let $C^{\circ}=\spf(\bc[\![x,y]\!]/(f(x,y)))$ be a germ of plane curve singularity. The \emph{Milnor number} and \emph{Tjurina number} of $C^{\circ}$ are defined respectively by
$$
\mu=\dim(\bc[\![x,y]\!]/(\partial_{x}f, \partial_{y}f)), \quad \tau=\dim(\bc[\![x,y]\!]/(f, \partial_{x}f, \partial_{y}f)).
$$
Let $A=\bc[\![x,y]\!]/(f(x,y))$ and $\widetilde{A}$ the normalization of $A$, the \emph{conductor} of $A$ is the ideal $$\ka=\big\{a\in A\mid a\widetilde{A}\subset A\big\}.$$ It is known that $\dim(\widetilde{A}/A)=\dim(A/\ka)$, we call it the \emph{$\delta$-invariant} of $C^{\circ}$, denoted $\delta(C^{\circ})$ or $\delta(A)$.

\paragraph{(4)} Let $C^{\circ}$ be a germ of plane curve singularity over $\bc$. Its irreducible components are called the \emph{branches}.  
It is known that every branch admits a \emph{Puiseux parametrization}
$$
\begin{cases}
X=t^{n},&\\
Y=\sum_{i=m}^{\infty} a_{i}t_{i}, &\text{for some }m>n.
\end{cases}
$$ 

\paragraph{(5)}

Let $C^{\circ}=\spf(A)$ be a germ of plane curve singularity over $\bc$. We denote by $\defm^{\rm top}_{C^{\circ}}$ or $\defm^{\rm top}_{A}$ the deformation functor of $C^{\circ}$. The functor $\defm^{\rm top}_{C^{\circ}}$ is known to be representable and we have a miniversal deformation of $C^{\circ}$. 
%Moreover, its tangent space $\defm^{\rm top}_{C^{\circ}}(\bc[\![\varep]\!])/\varep^{2}$ can be naturally identified with $\bc[\![x,y]\!]/(f, \partial_{x}f, \partial_{y}f)$. 

\paragraph{(6)}

Let $C_{0}$ be a projective algebraic curve over $\bc$ with planar singularities, its \emph{$\delta$-invariant} $\delta(C_{0})$ is defined as the sum of the $\delta$-invariants of its singularities. Let $\psi: C\to S$ be a deformation of $C_{0}$ with smooth base $S$, then the function $s\in S\mapsto \delta(C_{s})$ is upper semi-continuous according to Teissier \cite{teissier resolution}. The subscheme $S_{\delta}$ parametrizing curves with $\delta$-invariant $\delta$ is then locally closed, and we obtain a stratification $S=\bigsqcup_{\delta=0}^{\delta(C_{0})}S_{\delta}$, called the \emph{$\delta$-stratification} of $S$.

\paragraph{(7)}

We denote $\bn$ for the set of natural numbers and $\bn_{0}=\bn\cup \{0\}$. 

{\small
\paragraph*{Acknowledgement} 

We want to thank Prof. G\'erard Laumon for his constant support and encouragements during the years. Part of the work is done during the author's visit to the Research Centre for Mathematics and Interdisciplinary Sciences at Shandong University, we want to thank the institute for the wonderful hospitality.  

}

\section{The affine Springer fibers and the compactified Jacobians}\label{springer as cj}

Let $\gamma\in \kg[\![\varep]\!]$ be a regular semisimple element. Let $\rtt$ be the centralizer of $\gamma$ in $G$, then $\rtt$ acts on $\xx_{\gamma}$ by left translation. 
Let $\rS$ be the maximal $F$-split subtorus of $\rtt$.
Let $\Lambda\subset \rS(F)$ be the subgroup generated by $\chi(\varep),\,\chi\in X_{*}(\rS)$, then $\Lambda$ acts simply transitively on the irreducible components of $\xx_{\gamma}$. 
The connected components of $\xx_{\gamma}$ is naturally parametrized by $\bz$ with the morphism
$$
\xx_{\gamma}\to \bz,\quad [g]\mapsto \val(\det(g)).
$$
Let $\xx_{\gamma}^{0}$ be the central connected component of $\xx_{\gamma}$, then $\xx_{\gamma}\cong \xx_{\gamma}^{0}\times \bz$. 
Let $$
\Lambda^{0}=\{\lambda\in \Lambda\mid \val(\det(\lambda))=0\},
$$ 
then $\Lambda^{0}$ acts naturally on $\xx_{\gamma}^{0}$.
The quotient $\Lambda^{0}\backslash \xx_{\gamma}^{0}$ is a projective algebraic variety over $\bc$ and $\xx_{\gamma}^{0}\to \Lambda^{0}\backslash \xx_{\gamma}^{0}$ is an \'etale Galois covering of Galois group $\Lambda^{0}$.

\subsection{Laumon's work}

In an attempt to deform the affine Springer fibers, Laumon \cite{laumon springer} discovered a remarkable relationship between the affine Springer fibers and the compactified Jacobian of an algebraic curve. By definition, the affine Springer fiber $\xx_{\gamma}$ parametrizes the lattices $L$ in $F^{n}$ such that $\gamma L\subset L$. Such lattices can be naturally identified with sub-$\co[\gamma]$-modules of finite type in $E:=F[\gamma]\cong F^{d}$. Moreover, such $\co[\gamma]$-modules can be globalized.   
Let $C_\gamma$ be an irreducible projective algebraic curve over $\bc$, with two points $c$ and $\infty$ such that
\begin{enumerate}[nosep, label=(\roman*)]
\item
$C_\gamma$ has unique singularity at $c$ and $\widehat{\co}_{C_\gamma,\,c}\cong \co[\gamma]$,

\item

the normalization of $C_\gamma$ is isomorphic to $\bp^{1}$ by an isomorphism sending $\infty\in \bp^{1}$ to $\infty\in C_\gamma$.

\end{enumerate}

\noindent The curve $C_{\gamma}$ will be called the \emph{spectral curve} of $\gamma$.
Let $M$ be a sub-$\co[\gamma]$-modules of finite type in $F[\gamma]$. 
Similar to the obvious fact that $\co_{C_{\gamma}}$ can be obtained by gluing $\widehat{\co}_{C_{\gamma},c}=\co[\gamma]$ and ${\co}_{C_{\gamma}\backslash\{c\}}$,  we can glue $M$ with the sheaf $\co_{C_{\gamma}\backslash\{c\}}$ along $(C_{\gamma}\backslash\{c\})\times_{C_{\gamma}} \spec(\widehat{\co}_{C_{\gamma},c})=\spec(E)$, to get a torsion-free coherent $\co_{C_{\gamma}}$-module $\cm$ of generic rank $1$. 
Recall that the compactified Jacobian $\overline{P}_{C_{\gamma}}:=\overline{\pic}{}^{0}_{C_{\gamma}/\bc}$ parametrizes the isomorphism class of the torsion-free coherent $\co_{C_{\gamma}}$-modules of generic rank $1$ and degree $0$.
Consider the functor $\overline{P}{}_{C_\gamma}^{\natural}$ which associates to an affine $\bc$-scheme $S$ the groupo\"id of couples $(M,\iota)$, where $M$ is a torsion free coherent $\co_{C_\gamma\times S}$-module of generic rank $1$ and degree $0$, and $\iota:M|_{(C_\gamma\backslash c)\times S}\cong \co_{(C_\gamma\backslash c)\times S}$ is a trivialization of $M$ over $(C_\gamma\backslash c)\times S$. One verifies that $\overline{P}{}_{C_\gamma}^{\natural}$ is representable by a $\bc$-scheme. With the glueing construction, we get a natural morphism
$$
\xx_{\gamma}^{0}\to \overline{P}_{C_\gamma}^{\natural}, 
$$
which sends a lattice $L\subset E$ satisfying $\gamma L\subset L$ to the torsion-free $\co_{C_\gamma}$-module obtained by gluing $L$ and $\co_{C_\gamma\backslash \{c\}}$ along $\spec(E)$, together with the obvious trivialization. Forgetting the trivialization $\iota$, we get a morphism $\overline{P}_{C_\gamma}^{\natural}\to \overline{P}_{C_\gamma}$. For the affine Springer fiber, this corresponds to the quotient $\xx_{\gamma}^{0}\to \Lambda^{0}\backslash \xx_{\gamma}^{0}$.

\begin{thm}[Laumon \cite{laumon springer}]\label{homeo SJ}

The morphism $\xx_{\gamma}^{0}\to \overline{P}_{C_\gamma}^{\natural}$ is finite radical and surjective, and so is the induced morphism $\Lambda^{0}\backslash \xx_{\gamma}^{0}\to \overline{P}_{C_\gamma}$. In particular, both morphisms are universal homeomorphisms. 

\end{thm}

Unlike the affine Springer fiber or its quotients, the compactified Jacobian can be deformed naturally.
Let $\pi:(\cc, C_{\gamma})\to (\cb, 0)$ be an algebraization of a miniversal deformation of $C_{\gamma}$. 
Then $\overline{P}_{C_{\gamma}}$ fits in the projective flat family of relative compactified Jacobians 
$$
f:\overline{\cp}=\overline{\pic}{}^{\,0}_{\cc/ \cb}\to \cb.
$$ 
The compactified Jacobians of curves with planar singularities have very nice properties, which we can exploit to understand the affine Springer fibers.

On the other hand, it is more difficult to deform $\overline{P}{}_{C_{\gamma}}^{\natural}$. According to Laumon \cite{laumon springer}, it doesn't admit deformation over the whole base $\cb$ but only over the normalization $\widetilde{\cb}_{\delta_{\gamma}}$ of the $\delta$-stratum $\cb_{\delta_{\gamma}}$.
%(Recall that $\cb_{\delta}$ is the locally closed subscheme of $\cb$ consisting of the points $x\in \cb$ such that the $\delta$-invariant of $\cc_{x}$ is $\delta$). 
We will not need the full deformation of $\overline{P}{}_{C_{\gamma}}^{\natural}$ over $\widetilde{\cb}_{\delta_{\gamma}}$, but only its restriction to $\Delta_{\mu}$, where $\Delta_{\mu}$ is the closed subscheme of $\cb$ parametrizing equisingular deformations of $C_{\gamma}$ (see \S\ref{whitney regularity} for the notion of equisingularity).  According to Teissier \cite{zariski plane branch}, $\Delta_{\mu}$ is smooth over $\bc$.
%and it is isomorphic to its inverse image in $\widetilde{\cb}_{\delta_{\gamma}}$. 
Let $\overline{\cp}_{\mu}=\overline{\cp}\times_{\cb}\Delta_{\mu}$, and let $f_{\mu}:\overline{\cp}_{\mu}\to \Delta_{\mu}$ be the restriction of the family $f:\overline{\cp}\to \cb$ to $\Delta_{\mu}$.

\begin{thm}[Laumon \cite{laumon springer}]\label{universal in family}

There exists a $\Lambda^{0}$-torsor $\overline{\cp}{}^{\natural}_{\mu}\to \overline{\cp}_{\mu}$ such that its fiber over any geometric point $t\in \Delta_{\mu}$ coincides with the $\Lambda^{0}$-torsor $\overline{P}{}^{\natural}_{\cc_{t}}\to \overline{P}_{\cc_{t}}$.

\end{thm}

\begin{rem}\label{reduce main}

As we will see in \S\ref{whitney regularity}, for $\gamma,\gamma'\in \kg[\![\varep]\!]$ with the same root valuation datum, the spectral curves $C_{\gamma}$ and $C_{\gamma'}$ are equisingular. As a corollary of theorem \ref{homeo SJ} and \ref{universal in family},  theorem \ref{main CJ} implies theorem \ref{main}. 

\end{rem}

The glueing construction of Laumon works for curves with multiple planar singularities. We will not go into details here, but refer the reader to the original work of Laumon.

\subsection{Ng\^o's support theorem}

Consider the family $\pi:\cc\to \cb$ and its relative compactified Jacobians 
$
f:\overline{\cp}\to \cb.
$ 
According to Fantechi-G\"ottsche-van Straten \cite{fgv}, the total space $\overline{\cp}$ is smooth over $\bc$. 
As the morphism $
f:\overline{\cp}\to \cb
$
is projective, the complex $Rf_{*}\ql$ is pure by Deligne's Weil-II \cite{weil2}. The decomposition theorem of Beilinson, Bernstein, Deligne and Gabber \cite{bbdg} then implies that
$$
Rf_{*}\ql=\bigoplus_{n=0}^{2\delta_\gamma} {}^{p}\ch^{n}(Rf_{*}\ql)[-n],
$$
where $\delta_{\gamma}$ is the $\delta$-invariant of $C_{\gamma}$, which is also the relative dimension of the flat morphism $f$.

\begin{thm}[Ng\^o's support theorem \cite{ngo}]\label{ngo support}

We have decomposition
$$
Rf_{*}\ql=\bigoplus_{n=0}^{2\delta_\gamma} j_{!*}(R^{n}f^{\sm}_{*}\ql)[-n],
$$
where $j:\cb^{\rm sm}\to \cb$ is the natural inclusion. In particular,
$$
H^{n}\big(\overline{P}_{C_\gamma}, \ql\big)=\bigoplus_{i=0}^{n}\ch^{n-i}\big(j_{!*}{\textstyle{\bigwedge^{i}}}R^{1}f^{\sm}_{*}\ql\big)_{0},\quad n=0,\cdots,2\delta_\gamma.
$$
 
\end{thm}

\begin{rem}\label{ngo support extended}

We refer the reader to our work \cite{chen cj} for a sketch of the proof. Actually, the  theorem extends to the following situation: Let $B\to \cb$ be an embedding of a smooth algebraic variety over $\bc$, suppose that $B$ is transversal to all the $\delta$-strata $\cb_{\delta}$, i.e. for any point $x\in B\cap \cb_{\delta}$ the tangent space $T_{x}B$ is transversal to the tangent cone of $\cb_{\delta}$ at $x$. Then for the sub-family $f_{B}:\overline{\cp}\times_{\cb}B\to B$, Ng\^o's support theorem continues to hold. Indeed, by the local regularity criteria of Fantechi-G\"ottsche-van Straten \cite{fgv} (to be recalled below as theorem \ref{picard smooth local}), the total space $\overline{\cp}\times_{\cb}B$ is smooth over $\bc$. Moreover, Severi's inequality $\codim(B_{\delta})\ge \delta$ continues to hold by the transversality condition. We get then all the main ingredients in the proof of Ng\^o's support theorem.

\end{rem}

\section{Microlocal analysis of the intermediate extensions in Ng\^o's support theorem}

We will exploit Ng\^o's support theorem to understand the cohomology of $\overline{P}_{C_{\gamma}}$. For this, we make a microlocal analysis of the fiber $(j_{!*}R^{i}f_{*}^{\sm}\ql)_{0}$.
For simplicity we denote $\cf^{i}=R^{i}f_{*}^{\sm}\ql$.

The main ingredient for the analysis is the local regularity criteria of Fantechi-G\"ottsche-van Straten \cite{fgv}, we make a brief recall here. 
Let $C_{0}$ be an irreducible projective algebraic curve over $\bc$ with \emph{planar singularities}. For every $c\in C_{0}^{\rm sing}$, let 
$
R_{c}:=\widehat{\co}_{C_{0},c}\cong \bc[\![x,y]\!]/(f_{c}(x,y))
$ 
for some $f_{c}(x,y)\in \bc[\![x,y]\!]$, then the germ of singularity $\spf(R_{c})$ admits a miniversal deformation $\widehat{\phi}_{c}:\widehat{C}_{c}\to \widehat{B}_{c}$, and the tangent space of $\widehat{B}_{c}$ at its unique close point $0$ can be identified naturally with $\bc[\![x,y]\!]/(f_{c},\partial_{x}f_{c},\partial_{y}f_{c})$.
%Let $g_{c, i}(x,y), i=1,\cdots, \tau_{c},$ be a set of representative for the quotient $\bc[\![x,y]\!]/(f_{c},\partial_{x}f_{c},\partial_{y}f_{c})$, let 
%$$
%\widetilde{f}_{c}(x,y)=f_{c}(x,y)+\sum_{i=1}^{\tau_{c}} g_{c,i}(x,y)t_{i}\in \bc[\![x,y; t_{1},\cdots,t_{\tau_{c}}]\!], 
%$$
%then the family 
%$$
%\begin{tikzcd}
%V_{c}:=\big\{\widetilde{f}_{c}(x,y)=0\big\}  \arrow[r, hook] \arrow[d] & \spf(\bc[\![x,y; t_{1},\cdots,t_{\tau_{c}}]\!])\arrow[dl]\\
%\widehat{B}_{c}:=\spf(\bc[\![t_{1},\cdots,t_{\tau_{c}}]\!])&
%\end{tikzcd}
%$$
%gives a miniversal deformation of $\spf(\widehat{\co}_{C_{0},c})$. 
Let $EG_{c}$ be the closed sub formal scheme of $\widehat{B}_{c}$ which parametrizes the deformation of $\spf(R_{c})$ with constant geometric genus. According to Diaz-Harris \cite{dh}, the tangent cone of $EG_{c}$ at $0$ can be naturally identified with the sub vector space
$$
V(R_{c}):=\ka_{c}/(f_{c}, \partial_{x}f_{c}, \partial_{y}f_{c})\subset T_{0}(\widehat{B}_{c})=\bc[\![x,y]\!]/(f_{c},\partial_{x}f_{c},\partial_{y}f_{c}),
$$
where $\ka_{c}$ is the \emph{conductor} of $R_{c}$.
Note that $V(R_{c})$ is then of codimension $\delta(R_{c})$ in $T_{0}(\widehat{B}_{c})$. 
Let $\phi:(C, C_{0})\to (B,0)$ be a deformation of $C_{0}$ such that the restriction of $\phi$ to the formal neighbourhood $\widehat{B}_{0}$ of $B$ at $0$ is isomorphic to the deformation $\prod_{c\in C_{0}^{\rm sing}}\widehat{\phi}_{c}.$
Let $V(C_{0})$ be the union of the linear subspaces of $T_{0}(B)$ which is the preimage of 
$
V(R_{c})$
under the natural surjection $T_{0}(B)\cong \bigoplus_{c\in C_{0}^{\rm sing}}T_{0}(\widehat{B}_{c})\twoheadrightarrow T_{0}(\widehat{B}_{c})$ for all $c\in C_{0}^{\rm sing}$. 

\begin{thm}[Fantechi-G\"ottsche-van Straten \cite{fgv}]\label{picard smooth local}

Let $C_{T}\to T$ be a relative curve with base $T=\spec(\bc[\![t_{1},\cdots,t_{m}]\!])$ which comes from the family $\phi:C\to B$ by the base change with a morphism $T\to B$, then $\overline{\pic}{}^{0}_{C_{T}/T}$ is smooth over $\bc$ if and only if the image of the tangent space of $T$ is transversal to the subspace $V(C_{0})$ of $T_{0}(B)$. 

\end{thm}

\subsection{Calculation of some vanishing cycles}\label{morse groups}

By the general theory of stratification (cf. \cite{mather stratification}, proposition 10.1), for the family $\pi:\cc\to \cb$, let $\Sigma$ be the critical locus of $\pi$, then there exists a unique pair of Whitney stratifications $\big(\{\cc_{\alpha}\}_{\alpha\in \kaa}, \{\cb_{\beta}\}_{\beta\in \kbb}\big)$ of $\cc$ and $\cb$ respectively, such that $\Sigma$ is a union of strata in $\{\cc_{\alpha}\}_{\alpha\in \kaa}$, and that
\begin{enumerate}[topsep=1pt, noitemsep, label=$(\arabic*)$]
\item if $\cc_{\alpha}\subset \Sigma$, then $\pi|_{\cc_{\alpha}}$ is an immersion and $\pi(\cc_{\alpha})$ is a stratum in $\{\cb_{\beta}\}_{\beta\in \kbb}$,

\item

if $\cc_{\alpha}\cap \Sigma=\emptyset$, then $\pi(\cc_{\alpha})$ is a stratum in $\{\cb_{\beta}\}_{\beta\in \kbb}$ and $\cc_{\alpha}=\pi^{-1}(\pi(\cc_{\alpha}))\cap (\cc-\Sigma)$,

\item

all the other pairs of Whitney stratifications of $\cc$ and $\cb$ satisfying the above conditions are refinements of the above one.

\end{enumerate}

\noindent This unique pair of Whitney stratifications will be called the \emph{canonical} Whitney stratifications of the family $\pi:\cc\to \cb$.
By construction, $\cb^{\sm}$ is one of the strata in $\{\cb_{\beta}\}_{\beta\in \kbb}$, and for any stratum $\cc_{\alpha}$ with image $\cb_{\beta}$ the subspace $\cc\times_{\cb}\cb_{\beta}$ is Whitney regular along $\cc_{\alpha}$.  By theorem \ref{equisingularity}, the family $\pi:\cc\times_{\cb}\cb_{\beta}\to \cb_{\beta}$ is equisingular along $\cc_{\alpha}$. 
As $\cf^{i}=\bigwedge^{i}\cf^{1}$ and $\cf^{1}=R^{1}\pi^{\sm}_{*}\ql$ is lisse over $\cb^{\sm}$, the intermediate extension $j_{!*}\cf^{i}$ is constructible with respect to the Whitney stratification $\{\cb_{\beta}\}_{\beta\in \kbb}$.

Fix a stratum $\cb_{\beta}$ of dimension $d_{\beta}$, let $x$ be a point on it, we take a generic affine subspace of $\cb=\ba^{\tau_{\gamma}}$ of codimension $d_{\beta}$ passing through $x$ such that it  is transversal at $x$ to all the strata $\cb_{\beta}, \beta\in \kbb$.
This is possible because the stratification is locally finite. 
We then cut off from this affine subspace the closed subscheme where its intersection with the strata $\{\cb_{\beta}\}_{\beta\in \kbb}$ is not transversal, to get a dense open subscheme $V$ of this affine subspace. By construction, $V$ intersects transversally with all the $\delta$-strata and it contains $0$. 
Let $f_{V}:\overline{\cp}_{V}\to V$ be the restriction of the family $f:\overline{\cp}\to \cb$ to $V$.

\begin{prop}\label{support v}

The subspace $\overline{\cp}_{V}$ is nonsingular over $\bc$, and Ng\^o's support theorem holds for the family $f_{V}$, i.e.
$$
R(f_{V})_{*}\ql=\bigoplus_{i=0}^{2\delta_{\gamma}} (j_{V})_{!*} R^{i}(f^{\sm}_{V})_{*}\ql[-i],
$$
where $j_{V}:V^{\sm}\to V$ is the natural inclusion. 

\end{prop}

\begin{proof}

By construction, for any point $z\in V\cap \Delta$, $V$ is transversal to the subspace $V(\cc_{z})$ of $T_{z}(\cb)$.
By theorem \ref{picard smooth local}, $\overline{\cp}_{V}$ is nonsingular over $\bc$.
As a consequence, the complex $R(f_{V})_{*}\ql$ is pure. Moreover, Severi's inequality
$
\codim(V\cap \cb_{\delta})\ge \delta
$
continues to hold for all the $\delta$-strata as $V$ intersects transversally with them. Hence the reasoning of theorem \ref{ngo support} continues to work and we get the support theorem.
 
\end{proof}

Let $p:V\to \ba^{1}$ be a generic linear projection sending $x$ to $0$, such that its geometric fibers over a sufficiently small open neighborhood of $0\in \ba^{1}$ intersect transversally with the strata $\cb_{\beta}$ which contains $0$ in their closure. 
Let $\bar{\eta}_{0}$ be a geometric generic point of the strict Henselization $\ba^{1}_{\{0\}}$. Let $R\Psi_{\bar{\eta}_0}$  (resp. $R\Phi_{\bar{\eta}_0}$) be the nearby cycle (resp. vanishing cycle) functor with respect to the projection $p$.
We denote $\cf_{V}^{i}=R^{i}(f^{\sm}_{V})_{*}\ql$, it equals the restriction of $\cf^{i}$ to $V^{\sm}$.

%For any point $z\in \Delta_{\delta_{\gamma}-l+1}$, let $\phi_{z}:\widetilde{\cc}_{z}\to \cc_{z}$ be the normalization, it induces morphism
%$$
%\phi_{z}^{*}:\overline{\pic}_{\cc_{z}/k(z)}\to \pic_{\widetilde{\cc}_{z}/k(z)}.
%$$
%Let $\boldsymbol{\Lambda}_{z}^{\bullet}$ be the exterior algebra $\bigoplus H^{*}\big(\pic_{\widetilde{\cc}_{z}/k(z)},\ql\big)$. By remark \ref{cohom cj}, the cohomology group of $\overline{\pic}_{\cc_{z}/k(z)}$ is a free module over $\boldsymbol{\Lambda}_{z}^{\bullet}$.
%Note that $\pic_{\widetilde{\cc}_{z}/k(z)}$ is of dimension $l-1$.

\begin{thm}\label{local acyclicity}

Under the above settings, we have

\begin{enumerate}[topsep=0pt, itemsep=0pt, label=$(\arabic*)$]

\item
If $\cb_{\beta}$ is not one of the strict $\delta$-strata, then
$$
(R\Phi_{\bar{\eta}_{0}}R(f_{V})_{*}\bq_{\ell})_{x}=0.
$$
Consequently, for $i=1,\cdots,2\delta_\gamma$, we have
$$
(R\Phi_{\bar{\eta}_0}(j_{V})_{!*}\cf_{V}^{i})_{x}=0. 
$$

\item

If $\cb_{\beta}$ is one of the strict $\delta$-strata $\cb_{\delta}^{\circ}$, then
$$
(R\Phi_{\bar{\eta}_{0}}R(f_{V})_{*}\bq_{\ell})_{x}=\boldsymbol{\Lambda}_{x}^{\bullet}\otimes\ql(-\delta)[-2\delta+1],
$$
where $\boldsymbol{\Lambda}_{x}^{\bullet}$ is the exterior algebra $\bigoplus H^{*}\big(\jac_{\widetilde{\cc}_{x}},\ql\big)$ with $\widetilde{\cc}_{x}$ being the normalization of the curve $\cc_{x}$. 
Consequently,  we have
$$
(R^{\delta-1}\Phi_{\bar{\eta}_{0}}(j_{V})_{!*}\cf_{V}^{i})_{x}=
\boldsymbol{\Lambda}_{x}^{{i-\delta}}\otimes\ql(-\delta),
$$ 
and all the other vanishing cycles vanish.

\end{enumerate}

\end{thm}

\begin{proof}

Let $x'$ be a geometric point of $V_{\{x\}}\times_{\ba^{1}}\ba_{\{0\}}^{1}$, and $\cc_{x'}$ the fiber at $x'$ of the family $\pi:\cc\to \cb$. Let $\{c_{i}\}_{i\in I_{x'}}$ be the set of singular points of $\cc_{x'}$, let $R_{i}$ be the completed local ring of $\cc_{x'}$ at $c_{i}$. By the deformation theory of algebraic curves, we have a natural surjection of tangent spaces
\begin{equation}\label{tangent surj}
T_{x'}(\cb)\twoheadrightarrow \bigoplus_{i\in I_{x'}}\defm_{R_{i}}^{\rm top}(k[\ep]/\ep^{2}). 
\end{equation} 
Let $\tau_{i}$ be the Tjurina number of the local singularity at $c_{i}$, and $\delta_{i}$ the $\delta$-invariant. It is known that the tangent space $\defm_{R_{i}}^{\rm top}(k[\ep]/\ep^{2})$ is of dimension $\tau_{i}$ and that $\tau_{i}\ge \delta_{i}$.
%Moreover, the tangent cone $V(R_{i})$ of the $\delta_{i}$-stratum $\defm_{R_{i}}^{\rm top,\delta_{i}}$ at the point representing $R_{i}$ is of codimension $\delta_{i}$. 
%By theorem \ref{normalization in family}, we have 
%$$
%\delta(\cc_{x'})=\sum_{i\in I_{x'}} \delta_{i}\le \delta_{\gamma}. 
%$$

Let $H_{x'}$ be the fiber of $p$ passing through $x'$, it is of codimension 
$
d_{\beta}+1$. Note that $\cb_{\beta}$ is contained in the $\delta$-stratum with $\delta=\delta(\cc_{x})$, we get then the estimate 
\begin{equation}\label{codim d}
d_{\beta}\le \tau_{\gamma}-\delta(\cc_{x})\le \tau_{\gamma}-\delta(\cc_{x'})=\tau_{\gamma}-\sum_{i\in I_{x'}} \delta_{i},
\end{equation}
where the second inequality is due to the semicontinuity of the $\delta$-invariant (\cite{teissier resolution}, I-1.3.2). In particular, we have
\begin{equation}\label{dim h}
\dim(H_{x'})=\tau_{\gamma}-(d_{\beta}+1)\ge \sum_{i\in I_{x'}} \delta_{i}-1.
\end{equation}

$(1)$ In case that $\cb_{\beta}$ is not one of the strict $\delta$-strata, the first inequality in the equation (\ref{codim d}) is strict, hence also the inequality of (\ref{dim h}), and so
$$
\dim(H_{x'})\ge \sum_{i\in I_{x'}} \delta_{i}.
$$
 We claim that the image of $H_{x'}$ in $\defm_{R_{i}}^{\rm top}(\bc[\ep]/\ep^{2})$ under the projection (\ref{tangent surj}) is transversal to the subspace $V(R_{i})$ for all $i\in I_{x'}$.
For this, we need to verify that the morphism
\begin{equation}\label{tangent surj to H}
H_{x'}\to \bigoplus_{i\in I_{x'}}\defm_{R_{i}}^{\rm top}(k[\ep]/\ep^{2})/V(R_{i})
\end{equation}
is surjective. 
But this is clear because the vector space at the right hand side is a quotient of $T_{x'}\cb$ of dimension $\sum_{i\in I_{x'}}\delta_{i}$, and $H_{x'}$ is a generic subspace of $T_{x'}\cb$ of dimension at least $\sum_{i\in I_{x'}}\delta_{i}$.
With this transversality result, by theorem \ref{picard smooth local}, the subspace $(f_{V})^{-1}(H_{x'})$ is nonsingular at the points of the fiber $(f_{V})^{-1}(x')$. This implies that the composite $p\circ f_{V}$ is smooth at the points of $(f_{V})^{-1}(x)$ as it is clearly flat at these points. The local acyclicity of the smooth morphism then implies that
$$
R\Phi'_{\bar{\eta}_{0}}\bq_{\ell}\big|_{ (f_{V})^{-1}(x)}=0,
$$
where $R\Phi'_{\bar{\eta}_{0}}$ is the vanishing cycle functor for the morphism $p\circ f_{V}$. As $f_{V}$ is proper, by proper base change theorem, we get 
$$
(R\Phi_{\bar{\eta}_{0}}R(f_{V})_{*}\bq_{\ell})_{x}=0.
$$
With the decomposition of proposition \ref{support v}, we obtain 
\begin{equation}\label{vc decomp}
\big(R\Phi_{\bar{\eta}_0}R(f_{V})_{*}\bq_{\ell}\big)_{x}=\bigoplus_{i=0}^{2\delta_{\gamma}} \big(R\Phi_{\bar{\eta}_0}(j_{V})_{!*}\cf_{V}^{i}\big)_{x}[-i].
\end{equation}
By theorem \ref{ic via nearby}, we have 
$$
R^{q}\Phi_{\bar{\eta}_0}(j_{V})_{!*}\cf_{V}^{i}[\dim(V)]=
\begin{cases}
{\rm Im}\big\{T-\Id; R^{-1}\Psi_{\bar{\eta}_0}(j_{V})_{!*}\cf_{V}^{i}[\dim(V)]\big\},& \text{ for }q=-1,\\
0, &\text{ otherwise,}
\end{cases}
$$
where $T$ is the monodromy action on the nearby cycle. 
Hence 
\begin{equation}\label{vc decomp}
\big(R\Phi_{\bar{\eta}_0}R(f_{V})_{*}\bq_{\ell}\big)_{x}=\bigoplus_{i=0}^{2\delta_{\gamma}} \big(R^{\dim(V)-1}\Phi_{\bar{\eta}_0}(j_{V})_{!*}\cf_{V}^{i}\big)_{x}[-i].
\end{equation}
As the left hand side vanishes, so does each term on the right.

\medskip

$(2)$ In case that $\cb_{\beta}$ is one of the strict $\delta$-strata $\cb_{\delta}^{\circ}$, we get $\dim(V)=\delta$ and that the curve $\cc_{x}$ is irreducible with $\delta$ ordinary double points as its singularity.
We have then
$$
R_{i}\cong \bc[\![x,y]\!]/(xy),\quad \forall i\in I_{x},
$$
and so 
$$
\defm_{R_{i}}^{\rm top}\cong \spf(\bc[\![t]\!]), \quad \forall i\in I_{x},
$$
and we calculate directly that $V(R_{i})=0$. Moreover, the transversal slice $V$ can be identified with the base of a miniversal deformation of the singularities of $\cc_{x}$. 
Then the fiber $H_{x}$, being of dimension $\delta-1$, can not make the morphism (\ref{tangent surj to H}) surjective. But it does intersect transversally with each subspace $\defm_{R_{i}}^{\rm top}(\bc[\ep]/\ep^{2})$. 
Hence by theorem \ref{picard smooth local} all the singularities of the fiber $(p\circ f_{V})^{-1}(0)$ are smoothed when we deform it to the fiber $(p\circ f_{V})^{-1}(\bar{\eta}_{0})$. 
This deformation process can be understood more concretely as follows: Let $C_{xy}$ be the irreducible projective rational curve over $\bc$ with a unique ordinary double point as singularity, let $\overline{P}_{xy}$ be its compactified Jacobian, then $\overline{P}_{xy}\cong C_{xy}$.   
With the gluing construction of Laumon \cite{laumon springer}, we have the d\'evissage
\begin{equation}\label{devi j}
1\to \underbrace{\overline{P}_{xy}\times \cdots \times\overline{P}_{xy}}_{\delta} \to \overline{\jac}_{\cc_{x}}\to \jac_{\widetilde{\cc}_{x}}\to 1,
\end{equation}
and the deformation is only for the first factor. 
Let $\hat{f}_{xy}:\widehat{\cp}_{xy}\to \widehat{\ba}_{0}^{1}$ be the miniversal deformation of $\overline{P}_{xy}$. 
The deformation in question is then equivalent to
\begin{equation}\label{local model}
\underbrace{\widehat{\cp}_{xy}\times \cdots \times\widehat{\cp}_{xy}}_{\delta}
\xrightarrow{\hat{f}_{xy}\times\cdots\times \hat{f}_{xy}} \underbrace{\widehat{\ba}_{0}^{1}\times \cdots \times\widehat{\ba}_{0}^{1}}_{\delta}
\xrightarrow{\sum_{i=1}^{\delta}a_{i}} \widehat{\ba}_{0}^{1}.
\end{equation}
As $\overline{P}_{xy}\cong C_{xy}$, the deformation $\hat{f}_{xy}:\widehat{\cp}_{xy}\to \widehat{\ba}_{0}^{1}$ can be described explicitly as $\bc^{2}\to \bc: (x,y)\mapsto xy$,
and the above composite is nothing but
$$
\bc^{2\delta}\to \bc,\quad (x_{1}, y_{1}, \cdots,x_{\delta},y_{\delta})\mapsto \sum_{i=1}^{\delta}x_{i}y_{i},
$$
which describes the miniversal deformation of an ordinary quadratic singularity. Its vanishing cycle and the monodromy action is well known. 
Let $R\overline{\Phi}_{\bar{\eta}_{0}}$ be the vanishing cycle functor for the composite (\ref{local model}), then $R^{m}\bar{\Phi}_{\bar{\eta}_{0}}\ql$ is supported at the point $\bar{c}:=(c,\cdots,c)$ with $c$ being the unique double point of $\overline{P}_{xy}$, and that
$$
(R^{m}\bar{\Phi}_{\bar{\eta}_{0}}\ql)_{\bar{c}}
=\begin{cases} \ql(-\delta),& \text{if }m=2\delta-1,\\
0, &\text{otherwise}.
\end{cases}
$$
With the d\'evissage (\ref{devi j}) and the freeness property in \cite{ngo}, proposition 7.4.10, we obtain that 
$$
(R\Phi_{\bar{\eta}_{0}}R(f_{V})_{*}\ql)_{x}
%=\boldsymbol{\Lambda}_{x}^{\bullet}\otimes R\bar{\Phi}_{\bar{\eta}_{0}}\ql
=\boldsymbol{\Lambda}_{x}^{\bullet}\otimes \ql(-\delta)[-2\delta+1].
$$
With the same argument as in the first case, we have the decomposition (\ref{vc decomp}), hence
$$
\bigoplus_{i=0}^{2\delta_{\gamma}} \big(R^{\delta-1}\Phi_{\bar{\eta}_0}j_{V})_{!*}\cf_{V}^{i}\big)_{x}[-i]
=\boldsymbol{\Lambda}_{x}^{\bullet}\otimes \ql(-\delta)[-2\delta+1],
$$
where we have used the equality $\dim(V)=\delta$. Compare the terms of the same degree on both sides, we get the second assertion.

\end{proof}

\begin{rem}\label{local acyclicity intrinsic}

The results above are intrinsic to the singularities of $\cc_{x}$. Indeed, by the openness of versality, the restriction of the family $\pi:\cc\to \cb$ to the formal neighbourhood $\widehat{\cb}_{x}$ of $x$ in $\cb$ defines a versal deformation for the singularities of the fiber $\cc_{x}$.
More precisely, let $\{c_{i}\}_{i\in I_{x}}$ be the set of singularities of $\cc_{x}$, let $R_{i}=\widehat{\co}_{\cc_{x}, c_{i}}$ and $\widehat{\phi}_{i}:\widehat{C}_{i}\to \widehat{B}_{i}$ be the miniversal deformation of the singularity $\spf(R_{i})$.
Then there exists a smooth morphism $h: \widehat{\cb}_{x}\to \prod_{i\in I_{x}}\widehat{B}_{i}$ such that the family $\cc\times_{\cb}\widehat{\cb}_{x}\to \widehat{\cb}_{x}$ is the pull-back of the family $\prod_{i\in I_{x}}\widehat{\phi}_{i}$ along $h$.
For each family $\widehat{\phi}_{i}$, we have the canonical Whitney stratification and we can make transversal slice $V_{i}$ to the stratum containing the origin as in the theorem. Then the restriction of $\pi$ to the transversal slice $V$ is essentially the same as the product of the restriction of $\widehat{\phi}_{i}$ to $V_{i}$, similar factorization holds for $\cf^{1}\cong R^{1}\pi^{\sm}_{*}\ql$ and also the complex $j_{!*}\cf^{i}$ as $\cf^{i}=\bigwedge^{i}\cf^{1}$.

\end{rem}

\subsection{Reduction to the nearby fiber}\label{first reduction}

Following the general procedure described in the appendix \S\ref{iterated milnor}, we make a microlocal analysis of the fiber $(j_{!*}R^{i}f_{*}^{\sm}\ql)_{0}$.
Note that $\dim(\cb)=\tau_{\gamma}\ge \delta_{\gamma}$, and the equality holds if and only if the singularity $\spf(\co[\gamma])$ is a node, i.e. $\gamma$ is of the form $\gamma=\diag(\gamma_{1},\gamma_{2})\in \ggl_{2}(\co)$ with $\val(\gamma_{1}-\gamma_{2})=1$. 
Indeed, the $\delta$-stratum containing $0$ is of codimension $\delta_{\gamma}$ in $\cb$ according to Diaz-Harris \cite{dh}, hence in case $\tau_{\gamma}=\delta_{\gamma}$ the singularity $\spf(\co[\gamma])$ will be a node.
The geometry of $\xx_{\gamma}$ is clear for such elements, and we will therefore skip this case and assume that $\tau_{\gamma}>\delta_{\gamma}$. 
As there are only finitely many $\delta$-strata, the linear subspaces of $\cb=\ba^{\tau_{\gamma}}$ of dimension $\delta_{\gamma}+1$ which intersects transversally with the tangent cone at $0$ of all the $\delta$-strata form a dense open subvariety of the Grassmannian $\gr_{\delta_{\gamma}+1, \tau_{\gamma}}$.
Take such a linear subspace $H_{0}$ and cut off the locus where its intersection with the $\delta$-strata is not transversal, we get a dense open subscheme $\cb'$ of it. By construction, $\cb'$ intersects transversally with all the $\delta$-strata and it contains $0$. 
Let $\cc'=\cc\times_{\cb}\cb'$, consider the restriction 
$$
f':\overline{\cp}{}'=\overline{\pic}{}^{\,0}_{\cc'/\cb'}\to \cb'.
$$

\begin{prop}\label{support transverse}

The subspace $\overline{\cp}{}'$ is nonsingular over $\bc$, and Ng\^o's support theorem holds for the family $f'$, i.e.
\begin{equation*}
Rf'_{*}\ql=\bigoplus_{i=0}^{2\delta_\gamma} j'_{!*}(R^{\,i}{f'}^{\sm}_{*}\ql)[-i],
\end{equation*}
where $j':(\cb')^{\sm}\to \cb'$ is the natural inclusion. 

\end{prop}

\begin{proof}

By construction, for any point $z\in \cb'\cap \Delta$, $T_{z}(\cb')$ is transversal to the subspace $V(\cc_{x})$ of $T_{x}(\cb)$.
By theorem \ref{picard smooth local}, $\overline{\cp}'$ is nonsingular over $\bc$.
As a consequence, the complex $Rf'_{*}\ql$ is pure. To get the support theorem, note that Severi's inequality continues to hold as $\cb'$ intersects transversally with the $\delta$-strata, hence the reasoning of theorem \ref{ngo support} works in this case as well.
 
\end{proof}

Let $p:\cb'\to \ba^{1}$ be a generic projection. Let $\eta_{0}$ be the generic point of the strict Henselization $\ba^{1}_{\{0\}}$ of $\ba^{1}$ at $0$, and let $\bar{\eta}_{0}$ be a geometric point over it. Let $R\Psi_{\bar{\eta}_0}$  (resp. $R\Phi_{\bar{\eta}_0}$) be the nearby cycle (resp. vanishing cycle) functor with respect to the projection $p$.
We keep the notation $\cf^{i}$ for its restriction to $\cb'$.

\begin{prop}\label{red to nearby}

For $i=1,\cdots,2\delta_\gamma$, we have
$
(R\Phi_{\bar{\eta}_0}j'_{!*}\cf^{i})_{0}=0, 
$
and so
$$
(j'_{!*}\cf^{i})_{0}=
(R\Psi_{\bar{\eta}_0}j'_{!*}\cf^{i})_{0}, \quad i=0,\cdots, 2\delta_{\gamma}.
$$
Consequently, let $B_{0}(\ep)$ be a sufficiently small open ball around $0\in \cb=\ba^{\tau_{\gamma}}$ and $D_{0}(\ep')$   a sufficiently small open interval around $0\in \ba^{1}$, then 
\begin{equation*}
\ch^{*}(j'_{!*}\cf^{i})_{0}=H^{*}\big(\cb'\cap B_{0}(\ep)\cap p^{-1}(t), j'_{!*}\cf^{i}\big),\quad t\in D_{0}(\ep').
\end{equation*}

\end{prop}

\begin{proof}

By the choice of $p:\cb'\to \ba^{1}$, the tangent space at $0$ of its fiber at $0$ is generic of dimension $\delta_{\gamma}$ in $T_{0}(\cb)$, hence transversal to the codimension $\delta_{\gamma}$ subspace $V(C_{\gamma})$ of $T_{0}(\cb)$. 
By theorem \ref{picard smooth local}, the composite $p\circ f':\overline{\cp}'\to \ba^{1}$ is smooth at an \'etale neighborhood of $(f')^{-1}(0)$. By the local acyclicity of the smooth morphism, we have
$$
R\Phi'_{\bar{\eta}_{0}}\bq_{\ell}\big|_{ (f')^{-1}(0)}=0,
$$
where $R\Phi'_{\bar{\eta}_{0}}$ is the vanishing cycle functor for the morphism $p\circ f'$. As $f'$ is proper, by proper base change theorem, we get 
$$
(R\Phi_{\bar{\eta}_{0}}Rf'_{*}\bq_{\ell})_{0}=0.
$$
With the decomposition in proposition \ref{support transverse}, we deduce that $
(R\Phi_{\bar{\eta}_0}j'_{!*}\cf^{i})_{0}=0 
$
for all $i$ as in theorem \ref{local acyclicity}.
All the other assertions then follow. 

\end{proof}

\begin{cor}\label{reduce to nearby fiber}

With the above notations, for $i=0,\cdots, 2\delta_{\gamma}$, we have the decomposition
$$
H^{i}(\overline{P}_{C_{\gamma}}, \ql)=\bigoplus_{i'=0}^{i}H^{i-i'}\big(\cb'\cap B_{0}(\ep)\cap p^{-1}(t), j'_{!*}\cf^{i'}\big).
$$

\end{cor}

\begin{proof}

This is a corollary of proposition \ref{support transverse} and \ref{red to nearby}, and the equality
$H^{*}(\overline{P}_{C_{\gamma}}, \ql)=(R^{*}f'_{*}\ql)_{0}.
$

\end{proof}

\subsection{Iterated fiberations of the nearby fiber}\label{iterate}

We make further analysis of the cohomologies in corollary \ref{reduce to nearby fiber} by iterated fiberations of the nearby fiber as explained in the appendix \S\ref{iterated milnor}. 
%The general process is explained in the appendix \S\ref{iterated milnor}.
%Note that $\cb'$ is a dense open subvariety of a linear subspace $H_{0}$ of $\cb=\ba^{\tau_{\gamma}}$ and the local system $\cf^{i}$ has complicated behaviors along the discriminant $\cb'\cap \Delta$. Hence the iterated projections in \S\ref{iterated milnor} should be taken with respect to $\cb'\cap \Delta$. 
We take a generic flag of linear subspaces
$$
\{0\}= H_{\delta_{\gamma}+1}\subsetneq H_{\delta_{\gamma}}\subsetneq \cdots \subsetneq H_{2}\subsetneq H_{1}\subsetneq H_{0}, \quad \dim(H_{k})=\delta_{\gamma}+1-k \text{ for all }1\le k \le \delta_{\gamma}+1,
$$
for which the conclusions of proposition \ref{teissier local polar} and \ref{hm transversal} hold for the germs of $\cb'\cap \cb_{\beta}$  at $0$, $\beta\in \kbb$. This is possible as there are only finitely many such germs. Let $\pi_{k}:H_{0}\to \bc^{k}$ be the projection with kernel $H_{k}$. Let $t_{k}\in \bc^{k}$ be a  point sufficiently close to $0$, the projection $\pi_{k+1}$ induces an affine map 
$$
p_{k+1}:\cb'\cap B_{0}(\ep)\cap\pi_{k}^{-1}(t_{k})\to \bc^{1}.
$$ 
Its critical locus equals the intersection of $\pi_{k}^{-1}(t_{k})$ with the local polar variety $P_{0}(\cb'\cap B_{0}(\ep), H_{k+1})$ (see \S\ref{whitney regularity} for a brief review of local polar varieties), which is finite over $\bc$. The points $\{t_{k}\}_{k=1}^{\delta_{\gamma}+1}$ have been taken such that $\cb'\cap B_{0}(\ep)\cap \pi_{k+1}^{-1}(t_{k+1})$ coincides with the fiber of $p_{k+1}$ at a generic point in ${\rm Im}(p_{k+1})$. 
We use the iterated fibrations $p_{1},\cdots, p_{\delta_{\gamma}+1}$ to calculate the cohomology 
\begin{equation}\label{nearby to calculate}
\ch^{*}(j'_{!*}\cf^{i})_{0}=H^{*}\big(\cb'\cap B_{0}(\ep), j'_{!*}\cf^{i}\big).
\end{equation}
Note that the first fibration $p_{1}$ is just the projection $p$ in the previous section. By proposition \ref{red to nearby}, we have
\begin{equation*}
H^{*}\big(\cb'\cap B_{0}(\ep), j'_{!*}\cf^{i}\big)=H^{*}\big(\cb'\cap B_{0}(\ep)\cap p_{1}^{-1}(t_{1}), j'_{!*}\cf^{i}\big).
\end{equation*}
With the fibration $p_{2}:\cb'\cap B_{0}(\ep)\cap p_{1}^{-1}(t_{1})\to \bc^{1}$, the above cohomological group can be built up from 
\begin{equation*}
H^{*}\big(\cb'\cap B_{0}(\ep)\cap p_{2}^{-1}(t_{2}), j'_{!*}\cf^{i}\big)
\end{equation*}
and the vanishing cycles of the complex $j'_{!*}\cf^{i}$ with respect to $p_{2}$. By construction, the vanishing cycles are supported at the finite subscheme  $\pi_{1}^{-1}(t_{1})\cap P_{0}(\cb'\cap B_{0}(\ep), H_{2})$. This process can be iterated with the fibrations $p_{3}$, $p_{4}$ and so on until $p_{\delta_{\gamma}+1}$.

We can simplify the calculations with stratified Morse theory. Indeed, the ramification of $p_{k+1}:\cb'\cap B_{0}(\ep)\cap\pi_{k}^{-1}(t_{k})\to \bc^{1}$ can be very complicated. We replace the projection $p_{k+1}$ by a stratified Morse function 
$$
\varphi_{k+1}:\cb'\cap B_{0}(\ep)\cap\pi_{k}^{-1}(t_{k})\to \br^{1}
$$ 
with respect to the stratification of $\cb'\cap B_{0}(\ep)\cap\pi_{k}^{-1}(t_{k})$ induced from the  canonical Whitney stratification $\{\cb_{\beta}\}_{\beta\in \kbb}$, which approximates sufficiently well to the function ${\rm Re}(p_{k+1})$. Let $v_{k}$ be the smallest critical value of $\varphi_{k+1}$, then
$$
H^{*}\big(\cb'\cap B_{0}(\ep)\cap\pi_{k}^{-1}(t_{k})\cap \varphi_{k+1}^{-1}(-\infty, v_{k}),  j'_{!*}\cf^{i} \big)= H^{*}\big(\cb'\cap B_{0}(\ep)\cap p_{k+1}^{-1}(t_{k+1}), j'_{!*}\cf^{i}\big),
$$
and the cohomology group $H^{*}\big(\cb'\cap B_{0}(\ep)\cap\pi_{k}^{-1}(t_{k}),  j'_{!*}\cf^{i} \big)$ can be built up from it and the Morse groups at the critical points of $\varphi_{k+1}$. For simplicity, we denote $F^{k}=\cb'\cap B_{0}(\ep)\cap\pi_{k}^{-1}(t_{k})$.
%The Morse groups have been calculated in \S\ref{morse groups}. 

\begin{prop}\label{morse group}

Under the above setting, let $x\in F^{k}\cap\cb_{\beta}$ be a critical point of $\varphi_{k+1}$ with critical value $v$, let $\lambda$ be the Morse index of the restriction of $\varphi_{k+1}$ to $F^{k}\cap\cb_{\beta}$.
Let $a$ be a sufficiently small positive real number.

\begin{enumerate}[topsep=0pt, itemsep=0pt, label=$(\arabic*)$]

\item

If $\cb_{\beta}$ is not one of the strict $\delta$-strata, then
$$
H^{*}\big(F^{k}_{>v+a}, F^{k}_{<v-a}; j'_{!*}\cf^{i}\big)=0,\quad \text{for all }i.
$$

\item

If $\cb_{\beta}$ is one of the strict $\delta$-strata $\cb_{\delta}^{\circ}$, then for all $i$ we have
$$
H^{q}\big(F^{k}_{>v+a}, F^{k}_{<v-a}; j'_{!*}\cf^{i}\big)=
\begin{cases}
\boldsymbol{\Lambda}_{x}^{{i-\delta}}, & \text{ if }q=\lambda+\delta,
\\
0,&\text{otherwise}.
\end{cases}
$$

\end{enumerate}

\end{prop}

\begin{proof}

According to Goresky and MacPherson \cite{gm morse}, part II, 6.4 and the appendix 6.A, to calculate the Morse group, we take a transversal slice $N$ to $F^{k}\cap \cb_{\beta}$ in $F_{k}$ and a generic projection $p:N\to \bc^{1}$ sending $x$ to $0$, then the Morse group in question is equal to the vanishing cycle 
\begin{equation}\label{morse group to calculate}
\big(R^{d_{N}-1}\Phi_{\bar{\eta}_{0}} j'_{!*}\cf^{i}\big)_{x},
\end{equation} 
where $R\Phi_{\bar{\eta}_{0}}$ is the vanishing cycle for the projection $p$ at the point $0$ and $d_{N}$ is the dimension of $N$.  
Note that the projection $p_{k}$ and the point $t_{k}$ have been chosen generic, hence the fiber $F^{k}$ intersects transversally with the discriminant $\Delta$. 
Consequently, for any point $x'\in F^{k}\cap \Delta$, the restriction of $\pi:\cc\to \cb$ to the formal neighbourhood of $x'$ in $F^{k}$ defines a versal deformation of the singularities of $\cc_{x'}$ by the openness of versality. 
As explained in remark \ref{local acyclicity intrinsic}, this implies that the restriction of the complex $j'_{!*}\cf^{i}$ to the transversal slice $N$ depends only on the singularities of $\cc_{x}$, and it is essentially the same as the restriction of $j_{!*}\cf^{i}$ to the transversal slice $V$ in theorem \ref{local acyclicity}, as both can be reduced to transversal slices in the miniversal deformations of the singularities of $\cc_{x}$. In particular, the vanishing cycle (\ref{morse group to calculate}) equals that of theorem \ref{local acyclicity} and the assertion follows from the theorem.

\end{proof}

\begin{rem}\label{patch up morse groups}

In conclusion, the cohomology groups 
$$
H^{*}(F^{k}, j'_{!*}\cf^{i})=H^{*}\big(\cb'\cap B_{0}(\ep)\cap\pi_{k}^{-1}(t_{k}), j'_{!*}\cf^{i}\big), \quad k=1,\cdots, \delta_{\gamma}+1,
$$
can be calculated inductively, starting from $H^{*}(F^{\delta_{\gamma}+1}, j'_{!*}\cf^{i})$ and arriving at $$H^{*}(F^{1}, j'_{!*}\cf^{i})=H^{*}\big(\cb'\cap B_{0}(\ep)\cap\pi_{1}^{-1}(t_{1}), j'_{!*}\cf^{i}\big)=\ch^{*}(j'_{!*}\cf^{i})_{0}.
$$
With stratified Morse theory, the cohomology $H^{*}(F^{k}, j'_{!*}\cf^{i})$ can be built up from
$$
H^{*}\big(F^{k}\cap \varphi_{k+1}^{-1}(-\infty, v_{k}),  j'_{!*}\cf^{i} \big)=H^{*}(F^{k+1}, j'_{!*}\cf^{i})
$$
and the Morse groups $H^{*}\big(F^{k}_{>v+a}, F^{k}_{<v-a}; j'_{!*}\cf^{i}\big)$ at the critical values $v$ via the long exact sequence
$$
\cdots \to H^{*}\big(F^{k}_{<v-a}, j'_{!*}\cf^{i}\big) \to  H^{*}\big(F^{k}_{>v+a},  j'_{!*}\cf^{i}\big) \to H^{*}\big(F^{k}_{>v+a}, F^{k}_{<v-a}; j'_{!*}\cf^{i}\big) \to \cdots.
$$
To actually carry out this process, we still need to know how to attach
$$H^{q}\big(F^{k}_{>v+a}, F^{k}_{<v-a}; j'_{!*}\cf^{i}\big) \to H^{q+1}\big(F^{k}_{<v-a}, j'_{!*}\cf^{i}\big), \quad \text{for all }q,
$$ 
which seems to be out of reach now. 
Nonetheless, the above proposition tells us that the Morse groups are non-trivial only over the strict $\delta$-strata, which is enough for our purpose.

\end{rem}

\section{Whitney regularity along the equisingular strata}

Based on a study of the geometry of the discriminant $\Delta$, we will show that the union of the strict $\delta$-strata is Whitney regular over the equisingular strata. In conjunction with the microlocal analysis of the complexes $j_{!*}\cf^{i}$, this implies that the sheaves $\ch^{*}(j_{!*}\cf^{i})$ are locally constant over the equisingular strata.

\subsection{The geometry of the discriminant}

Teissier \cite{teissier hunting} has discovered a very remarkable property of the discriminant of the miniversal deformation of isolated hypersurface singularities, which is further extended to the isolated complete intersection singularities by Looijenga \cite{looij icis}. We make a recount of their result.

Let $(X_{0},0)\subset (\bc^{n}, 0)$ be a germ of isolated complete intersection singularity.
Let $\defm_{(X_{0},0)\hookrightarrow (\bc^{n},0)}$ be the deformation functor for the embedding $(X_{0},0)\hookrightarrow (\bc^{n},0)$, and let $\defm_{(X_{0},0)}$ be the deformation functor for the germ $(X_{0}, 0)$. 
Let $T^{1}_{(X_{0},0)\hookrightarrow (\bc^{n},0)}$ and $T_{(X_{0},0)}^{1}$ be their tangent spaces, i.e.
$$
T^{1}_{(X_{0},0)\hookrightarrow (\bc^{n},0)}=\defm_{(X_{0},0)\hookrightarrow (\bc^{n},0)}(\bc[\varep]/\varep^{2}), \quad T_{(X_{0},0)}^{1}=\defm_{(X_{0},0)}(\bc[\varep]/\varep^{2}).
$$
Let $I_{0}\subset \co_{\bc^{n}, 0}$ be the defining ideal of $X_{0}$, then its conormal sheaf $I_{0}/I_{0}^{2}$ is locally free, and we get an exact sequence
\begin{equation}\label{cotan exact}
0\to I_{0}/I_{0}^{2}\to \Omega_{\bc^{n}}^{1}\otimes_{\co_{\bc^{n}}}\co_{X_{0}}\to \Omega_{X_{0}}^{1}\to 0. 
\end{equation}
Let $\cn_{X_{0}/\bc^{n}}:=\ch om_{\co_{X_{0}}}(I_{0}/I_{0}^{2}, \co_{X_{0}})$, we call it the normal sheaf of $X_{0}$. 
According to the general deformation theory (cf. \cite{greuel}), we can identify the tangent space 
$$
T^{1}_{(X_{0},0)\hookrightarrow (\bc^{n},0)}=\cn_{X_{0}/\bc^{n}, 0}, 
$$
and the tangent space $T^{1}_{(X_{0},0)}$ fits in the exact sequence
\begin{equation}\label{tan exact}
0\to \Theta_{X_{0},0}\to \Theta_{\bc^{n}, 0}\otimes_{\co_{\bc^{n},0}}\co_{X_{0},0} \to \cn_{X_{0}/\bc^{n}, 0}\to T^{1}_{(X_{0},0)}\to 0,
\end{equation}
which is obtained from the exact sequence (\ref{cotan exact}) by taking dual. Note that this leads to the isomorphism
\begin{equation*}
T^{1}_{(X_{0},0)}\cong \ce xt_{\co_{X_{0},0}}^{1}(\Omega_{X_{0},0}^{1}, \co_{X_{0},0}).
\end{equation*} 
Let $f=(f_{1},\cdots,f_{k})$ be a minimal set of generators for the defining ideal $I_{0}$ of $X_{0}$, let $Df$ be the Jacobian matrix
$$
Df=\begin{bmatrix}
\frac{ \partial f_{1}}{\partial x_{1}}&\cdots&\frac{ \partial f_{1}}{\partial x_{n}}\\
\vdots &\ddots & \vdots \\
\frac{ \partial f_{k}}{\partial x_{1}} &\cdots& \frac{ \partial f_{k}}{\partial x_{n}}
\end{bmatrix}.
$$
With the exact sequence (\ref{tan exact}), we calculate
\begin{equation*}
T_{(X_{0}, 0)}^{1}=\co_{\bc^{n}, 0}^{ k}/(Df\cdot \co_{\bc^{n}, 0}^{ n}+I_{0}\cdot \co_{\bc^{n},0}^{ k}).
\end{equation*}
The $\co_{X_{0},0}$-module at the right hand side is traditionally called \emph{Tjurina module}.
Let $g_{1},\cdots, g_{\tau}\in \co_{\bc^{n}, 0}^{ k}$, $g_{i}=(g_{i, 1}, \cdots, g_{i,k})^{t}$, be a set of representatives for a basis of the above quotient as a vector space over $\bc$.
Let
$$
F_{i}(\mathbf{x}, \mathbf{t})=f_{i}(x_{1},\cdots, x_{n})+\sum_{j=1}^{\tau} t_{j} g_{i,j}(x_{1},\cdots,x_{n}),\quad i=1,\cdots, k.
$$ 
Let $I\subset \co_{\bc^{n}\times\bc^{\tau}}$ be the ideal generated by $F_{1}, \cdots, F_{k}$, and let $X$ be the closed subscheme of $\bc^{n}\times \bc^{\tau}$ defined by $I$. With the second projection $\bc^{n}\times \bc^{\tau}\to \bc^{\tau}$, we get a family 
$$
\psi: (X, 0)\to (S=\bc^{\tau},0),
$$ 
which is a miniversal deformation of $X_{0}$. One verifies that $X$ is nonsingular. 
Let $\Sigma\subset X$ be the \emph{critical locus} of the morphism $\psi$, i.e. it is the closed subscheme defined by the Fitting ideal ${\rm F}_{n-k}(\Omega^{1}_{X/S})$.  Then all the geometric fibers of $\psi$ have isolated complete intersection singularities at its intersection with $\Sigma$. 
In concrete terms, $\Sigma$ is the closed subscheme of $X$ defined by the condition that the Jacobian matrix
$$
\begin{bmatrix}
\frac{ \partial F_{1}}{\partial x_{1}}&\cdots&\frac{ \partial F_{1}}{\partial x_{n}}\\
\vdots &\ddots & \vdots \\
\frac{ \partial F_{k}}{\partial x_{1}} &\cdots& \frac{ \partial F_{k}}{\partial x_{n}}
\end{bmatrix}
$$
is of rank $<k$, i.e. all the minors of the Jacobian matrix of size $k\times k$ vanish.

At each geometric point $x\in \Sigma$, let $s=\psi(x)$, we get the Tjurina module $T^{1}_{(X_{s},x)}$. They can be put in family as $x$ varies. We start from the exact sequence  
\begin{equation*}
0\to \psi^{*}\Omega_{S}^{1}\to \Omega_{X}^{1}\to \Omega_{X/S}^{1}\to 0. 
\end{equation*}
Take its dual, we get
\begin{equation}\label{tan exact f}
\cdots\to \Theta_{X}\to \psi^{*}\Theta_{S} \to  \ce xt_{\co_{X}}^{1}(\Omega_{X/S}^{1}, \co_{X}) \to 0.
\end{equation}
Let $\ct_{X/S}^{1}=\ce xt_{\co_{X}}^{1}(\Omega_{X/S}^{1}, \co_{X})$, then 
\begin{equation}\label{base change tjurina}
\ct_{X/S}^{1}\otimes_{\co_{X}}\co_{X_{s},x}
=\ce xt_{\co_{X_{s},x}}^{1}(\Omega_{\co_{X_{s},x}}^{1}, \co_{\co_{X_{s},x}})
=T^{1}_{(X_{s},x)}.
\end{equation}
Let $F=(F_{1},\cdots,F_{k})^{t}\in \co^{k}_{\bc^{n}\times \bc^{\tau},0}$, and let $DF$ be its Jacobian matrix with respect to the variables $x_{1},\cdots, x_{n}$. The exact sequence (\ref{tan exact}) can be put in families as well, and from it we calculates
$$
\ct_{X/S}^{1}=\co_{\bc^{n}\times\bc^{\tau}, 0}^{ k}/(DF\cdot \co_{\bc^{n}\times\bc^{\tau}, 0}^{ n}+I\cdot \co_{\bc^{n}\times\bc^{\tau},0}^{ k}).
$$
The $\co_{X}$-module at the right hand side is called the \emph{relative Tjurina module} of the family $\psi:X\to S$.

Actually the above construction of $\ct_{X/S}^{1}$ works for any deformation $\psi':X'\to S'$ of $X_{0}$. From the analogue of the exact sequence (\ref{tan exact f}), for any geometric point $x\in X'$ with image $s\in S'$, we deduce a morphism 
$$
\Theta_{S',s}\xrightarrow{\psi^{-1}} (\psi^{*}\Theta_{S'})_{x} \to \ct^{1}_{X'/S',x}.
$$
Modulo the maximal ideal $\km_{S',s}$ of $\co_{S',s}$, we get the \emph{Kodaira-Spencer} map
$$
\kappa_{x}: T_{s}S'\to T_{(X'_{s},x)}^{1} 
$$
It is known that $\kappa_{x}$ is surjective if $\psi'$ is versal and isomorphic if $\psi'$ is miniversal (cf. \cite{greuel}).

Let $\Delta\subset S$ be the image of $\Sigma$, called the \emph{discriminant} of the morphism  $\psi$, it is the closed subscheme defined by the Fitting ideal ${\rm F}_{0}(\psi_{*}\co_{\Sigma})$.

\begin{thm}[Purity of the discriminant, Teissier \cite{teissier hunting}]

Let $\psi:(X,0)\to (S,0)$ be a miniversal deformation of an isolated complete intersection singularity as above. Then $\Delta\subset (S,0)$ is a hypersurface which is locally analytically irreducible at $0$, and the induced morphism $\Sigma\to \Delta$ is the normalization map.

\end{thm}

\begin{proof}

We give a sketch of the proof.
From the above concrete description of $\psi$, one deduce that $X$ is regular and it is of dimension $n-k+\tau$, and that the restriction $\psi:\Sigma\to \Delta$ is a finite morphism. In particular, $\dim(\Sigma)=\dim(\Delta)$. As the generic fiber of $\psi$ is smooth, we obtain that $\dim(\Delta)\le \tau-1$. 
On the other hand, for any geometric point $x\in X$, take a local coordinate of $X$ at $x$, the tangent map induces a morphism $\rd \psi_{x}: (X,x)\to \Hom(\bc^{n-k+\tau}, \bc^{\tau})$. 
For $m,d\in \bn$ and $r=0, \cdots,\min(m,d)$, let 
$$
\co_{r}(m,d)=\big\{A\in \Hom(\bc^{m}, \bc^{d})\mid \rk(A)=r\big\}. 
$$
Consider the action of $\gl_{m}\times\gl_{d}$ on $\Hom(\bc^{m}, \bc^{d})$, one verifies that $\co_{r}(m,d)$ is a smooth irreducible affine algebraic variety of codimension $(m-r)(d-r)$ in $\Hom(\bc^{m}, \bc^{d})$, and that the Zariski closure satisfies $\overline{\co}_{r}(m,d)=\bigsqcup_{r'=0}^{r}\co_{r'}(m,d)$. 
%For $m,d\in \bn$ with $m>d$, let $\cn_{m,d}$ be the closed subscheme of $\Hom(\bc^{m}, \bc^{d})$ consisting of the morphisms of rank $<d$. 
%On verifies that $\cn_{m,d}$ is irreducible of codimension $m-d+1$\footnote{A quick way to verify this fact is to use the \emph{Springer resolution} of $\cn_{m,d}$: Let $\widetilde{\cn}_{m,d}=\big\{(M, H)\in \Hom(\bc^{m}, \bc^{d})\times \check{\bp}^{d}\mid {\rm Im}(M)\subset H\big\}$, where $\check{\bp}^{d}$ parametrizes the hyperplanes in $\bc^{d}$. With the projection $\widetilde{\cn}_{m,d}\to \check{\bp}^{d}$, one verifies that $\widetilde{\cn}_{m,d}$ is smooth irreducible of dimension $(m+1)(d-1)$. Hence the projection $\widetilde{\cn}_{m,d}\to \cn_{m,d}$ is a resolution of singularity as it is generically isomorphic.}. 
Then the germ $(\Sigma,x)$ is the inverse image of $\overline{\co}_{\tau-1}(n-k+\tau,\tau)$ under the morphism $\rd \psi_{x}$. 
Hence the codimension of $(\Sigma,x)$ in $(X, x)$ is at most $n-k+1$, i.e. $\dim(\Sigma)\ge \tau-1$. Combining the two inequalities, we obtain that both $\Sigma$ and $\Delta$ are of pure dimension $\tau-1$. In particular, $\Delta$ is a hypersurface in $(\bc^{\tau},0)$.

On verifies then that $\Sigma$ is Cohen-MaCaulay. 
With the concrete description of $\psi$, one verifies that ${\rm F}_{n-k}(\Omega_{X/S}^{1})={\rm F}_{0}(\ct_{X/S}^{1})$.
As $X$ is regular and $\Sigma$ is of pure codimension $n-k+1$, we have
\begin{equation}\label{depth equality}
{\rm depth}\big({\rm F}_{0}(\ct_{X/S, x}^{1}), \co_{X,x}\big)=n-k+1. 
\end{equation}
By definition, $\ct_{X/S}^{1}$ has presentation
$$
\co_{X}^{n}\xrightarrow{DF} \co_{X}^{k}\to \ct_{X/S}^{1}\to 0.  
$$
According to Buchsbaum-Rim \cite{br}, Corollary 2.7, the equality (\ref{depth equality}) implies  that the homological dimension ${\rm hd}_{\co_{X,x}}(\co_{\Sigma,x})$ is $n-k+1$. As $X$ is regular, we have the equality of Auslander-Buchsbaum
$$
{\rm hd}_{\co_{X,x}}(\co_{\Sigma,x})+{\rm depth}_{\co_{X,x}}(\co_{\Sigma,x})=\dim(X),
$$
whence ${\rm depth}_{\co_{X,x}}(\co_{\Sigma,x})=\tau-1=\dim(\Sigma)$ and so $\Sigma$ is Cohen-MaCaulay.

By Serre's criteria of normality, we need to verify that the singular locus of $\Sigma$ is of codimension at least $2$. For this we introduce the \emph{Thom strata}: Let $\Theta_{X}$ and $\Theta_{S}$ be the sheaves of holomorphic vector fields on $X$ and $S$ respectively, the pull-back and the tangent map induces the morphisms of sheaves
$$
\psi^{-1}: \Theta_{S}\to \psi^{*}(\Theta_{S})\quad\text{and}\quad t\psi: \Theta_{X}\to \psi^{*}\Theta_{S}.
$$

\begin{lem}\label{inf stable}
The miniversal deformation $\psi$ is \emph{infinitesimally stable}, i.e. for any point $x\in X$ with image $s$, we have
\begin{equation*}
\psi^{*}(\Theta_{S, s})=\psi^{-1} (\Theta_{S, s})+t\psi(\Theta_{X,x}).
\end{equation*}

\end{lem}

\begin{proof}

Let $\km_{x}\subset \co_{X, x}$ be the maximal ideal at $x$ and $\km_{s}\subset \co_{S,s}$ the maximal ideal at $s$. We identify $\km_{s}$ with its image in $\co_{X,x}$ via the morphism $\co_{S,s}\to \co_{X,x}$. 
We claim that the assertion is equivalent to 
\begin{equation}\label{inf stable 1}
\psi^{*}(\Theta_{S, s})=\psi^{-1} (\Theta_{S, s})+t\psi(\Theta_{X,x})+\km_{s}\psi^{*}(\Theta_{S, s}).
\end{equation}
Indeed, the above equation is equivalent to
$$
\frac{\psi^{*}(\Theta_{S, s})}{t\psi (\Theta_{X,x})}=\frac{\psi^{-1} (\Theta_{S, s})+t\psi(\Theta_{X,x})}{t\psi(\Theta_{X,x})}+\frac{\km_{s}\psi^{*}(\Theta_{S, s})+t\psi(\Theta_{X,x})}{t\psi(\Theta_{X,x})}. 
$$
Let $M=\frac{\psi^{*}(\Theta_{S, s})}{t\psi (\Theta_{X,x})}$, the equation can be rewritten as
\begin{equation}\label{inf stable 2}
M=\frac{\psi^{-1} (\Theta_{S, s})+t\psi(\Theta_{X,x})}{t\psi(\Theta_{X,x})}+\km_{s}M.
\end{equation}
One verifies that $M$ is an $\co_{S,s}$-module of finite type. By Nakayama's lemma, the above equation is equivalent to
$$
M=\frac{\psi^{-1} (\Theta_{S, s})+t\psi(\Theta_{X,x})}{t\psi(\Theta_{X,x})},
$$
which is exactly the equation in the lemma.

For the equation (\ref{inf stable 1}) or equivalently (\ref{inf stable 2}), it is enough to verify the surjectivity of 
\begin{equation}\label{inf stable 3}
T_{s}S\xrightarrow{\psi^{-1}} M/\km_{s}M.
\end{equation}
Note that
$$
M/\km_{s}M=\frac{\psi^{*} (\Theta_{S, s})}{\km_{y}\psi^{*} (\Theta_{S, s})+t\psi(\Theta_{X,x})}=T_{(X_{s},x)}^{1},
$$
and that the morphism (\ref{inf stable 3}) is just the Kodaira-Spencer map, here we have used the exact sequence (\ref{tan exact f}) and the isomorphism (\ref{base change tjurina}). Now the surjectivity of (\ref{inf stable 3}) follows from the versality of $\psi$.

\end{proof}

According to Mather \cite{mather stable 2}, the morphism $\psi$ is \emph{stable}, i.e. it is analytically equivalent to any sufficiently small perturbation of it. This implies that the tangent map $\rd\psi_{x}:(X, x)\to \Hom(\bc^{n-k+\tau}, \bc^{\tau})$ is transversal to the stratification $\Hom(\bc^{n-k+\tau}, \bc^{\tau})=\bigsqcup_{r=0}^{\tau}\co_{r}(n-k+\tau, \tau)$. 
Consequently, for $r=0,\cdots, \tau-1$, the inverse image
$$
\Sigma_{r,x}=(\rd \psi_{x})^{-1}\big(\co_{r}(n-k+\tau, \tau)\big)
$$ 
is smooth. $\Sigma_{r,x}$ is clearly independent of the local coordinate of $X$ at $x$, hence when $x$ runs through points on $X$  they glue together to a smooth locally closed subscheme $\Sigma_{r}$ of $X$, called \emph{Thom strata}.
It is clear that $\Sigma$ equals the Zariski closure of $\Sigma_{\tau-1}$. As $\Sigma_{\tau-1}$ is itself smooth, the singularity of $\Sigma$ is contained in the closure of $\Sigma_{\tau-2}$, which is of codimension $2(n-k+2)-(n-k+1)=n-k+3\ge 3$ in $\Sigma$. This finishes the proof that $\Sigma$ is normal.

As the fibers of $\psi$ have at worst isolated complete intersection singularities and the singularities can be deformed independently, the fibers of $\psi$ over the generic point of each local irreducible component of $(\Delta,0)$ can have only a unique singularity. 
Moreover, for any point $x$ of $\Sigma_{\tau-1}$, one verifies by the openness of versality that the singularity $(X_{\psi(x)}, x)$ must be ordinary quadratic.
With the deformation theory of ordinary quadratic singularity, the restriction of $\psi$ to $\Sigma_{\tau-1}$ must be analytically locally an isomorphism. 
Hence $\psi$ is generically an isomorphism over each local irreducible component of $(\Delta, 0)$.
%Moreover, let $x$ be a generic geometric point of $\Sigma_{\tau-1}$ and let $L$ be a transversal slice to $\Delta$ at $\psi(x)$, then the restriction of $\psi$ to the inverse image $(\psi^{-1}(L), x)$ is a miniversal deformation of the singularity $(X_{\psi(x)}, x)$ by the openness of versality. Note that $L$ is smooth of dimension $1$, this implies that the singularity $(X_{\psi(x)}, x)$ must be ordinary quadratic. 
The normality of $\Sigma$ implies that it is locally analytically irreducible, this implies that $\Delta$ is locally analytically irreducible at $0$ as well.  
As $\psi$ is finite generically an isomorphism and $\Sigma$ is normal, $\psi$ must be the normalization.

\end{proof}

The Thom stratification inspires the following construction: 
For $m,d\in \bn$ with $m>d$, let 
$$
\cn_{m,d}=\big\{A\in \Hom(\bc^{m}, \bc^{d})\mid \rk(A)<d\big\}, 
$$
which is also the Zariski closure of ${\co_{d-1}(m,d)}$. It admits a resolution called \emph{Springer resolution}:
Let $\widetilde{\cn}_{m,d}=\big\{(M, H)\in \Hom(\bc^{m}, \bc^{d})\times \check{\bp}^{d-1}\mid {\rm Im}(M)\subset H\big\}$, where $\check{\bp}^{d-1}$ parametrizes the hyperplanes in $\bc^{d}$. With the projection $\widetilde{\cn}_{m,d}\to \check{\bp}^{d-1}$, one verifies that $\widetilde{\cn}_{m,d}$ is smooth irreducible. Hence the projection $\widetilde{\cn}_{m,d}\to \cn_{m,d}$ is a resolution of singularity as it is proper and generically isomorphic. We define $\widetilde{\Sigma}$ with the Cartesian diagram
$$
\begin{tikzcd}
\widetilde{\Sigma}\arrow[r] \arrow[d]\arrow[dr, phantom, "\square"]& \widetilde{\cn}_{n-k+\tau,\tau}\arrow[d]\\
X\arrow[r, "\rd\psi"]&\Hom(\bc^{n-k+\tau}, \bc^{\tau}).
\end{tikzcd}
$$
As $X$ is transversal to the stratification of $\Hom(\bc^{n-k+\tau}, \bc^{\tau})$ by the rank, the fiber product $\widetilde{\Sigma}$ is smooth and the resulting morphism $\widetilde{\Sigma}\to \Sigma$ is a resolution of singularity. More intrinsically, 
$$
\widetilde{\Sigma}=\big\{(x, H)\in \Sigma\times \check{\bp}^{\tau-1}\mid {\rm Im}{(\rd\psi_{x})}\subset H\big\},
$$
where $\check{\bp}^{\tau-1}$ parametrizes the hyperplanes in $\bc^{\tau}$ and we have moved ${\rm Im}{(\rd\psi_{x})}$ to the plane parallel to it passing through $0$.

On the other hand, let $\widetilde{\Delta}$ be the Nash blow-up of $\Delta$, i.e. it is the Zariski  closure of the graph $\widetilde{\Delta}^{\reg}$ of the embedding
$$
\Delta^{\reg}\hookrightarrow \Delta\times \check{\bp}^{\tau-1},\quad y\mapsto (y, T_{y}\Delta),
$$
where $\Delta^{\reg}$ is the regular locus of $\Delta$.
Note that there exists a natural morphism $\widetilde{\psi}:\widetilde{\Sigma}\to \widetilde{\Delta}$. 
Indeed, the composite $\widetilde{\Sigma}\hookrightarrow \Sigma\times \check{\bp}^{\tau-1}\to \Delta\times\check{\bp}^{\tau-1}$ is clearly proper, hence its image in $\Delta\times \check{\bp}^{\tau-1}$ is closed. It is even irreducible as $\widetilde{\Sigma}$ is irreducible. 
Let $\widetilde{\Sigma}_{\tau-1}$ be the inverse image of ${\Sigma}_{\tau-1}$ for the projection $\widetilde{\Sigma}\to \Sigma$.
For any point $x\in \Sigma_{\tau-1}$ such that $\psi(x)\in \Delta^{\reg}$, the singularity of the fiber $X_{\psi(x)}$ at $x$ is ordinary quadratic, hence the rank of $\rd\psi_{x}$ is $\tau-1$ and ${\rm Im}(\rd\psi_{x})=T_{\psi(x)}\Delta^{\reg}$. This implies that the image of $\widetilde{\Sigma}_{\tau-1}$ under the morphism $\widetilde{\Sigma}\to \Delta\times\check{\bp}^{\tau-1}$ contains $\widetilde{\Delta}^{\reg}$. 
As the image of $\widetilde{\Sigma}\to \Delta\times\check{\bp}^{\tau-1}$ is closed irreducible and   $\widetilde{\Delta}^{\reg}$ is dense in $\widetilde{\Delta}$, we obtain that $\widetilde{\Sigma}$ does map onto $\widetilde{\Delta}$ and this defines $\widetilde{\psi}$.

\begin{thm}[Teissier \cite{teissier hunting} for the hypersurface, Looijenga \cite{looij icis}] \label{nash nature}

The morphism $\widetilde{\psi}:\widetilde{\Sigma}\to \widetilde{\Delta}$ is an isomorphism. In particular, the Nash blow-up of $\Delta$ is smooth.

\end{thm}

\begin{proof}

As $\widetilde{\psi}$ is already known to be surjective, it is enough to show that it is an immersion at any point $\tilde{x}\in \widetilde{\Sigma}$. Let $\tilde{s}=\widetilde{\psi}(\tilde{x})\in \widetilde{\Delta}$, we need to show that $\widetilde{\psi}^{*}\km_{\widetilde{\Delta},\tilde{s}}=\km_{\widetilde{\Sigma}, \tilde{x}}$. 
The inclusion $\widetilde{\psi}^{*}\km_{\widetilde{\Delta},\tilde{s}}\subset\km_{\widetilde{\Sigma}, \tilde{x}}$ is clear and it is enough to show the inclusion in the other direction. 
Let $\tilde{x}=(x,H)\in \Sigma\times \check{\bp}^{\tau-1}$ and let $s=\psi(x)$, then $\tilde{s}=(s, H)$.
%As $\Delta$ is a germ of reduced irreducible hypersurface in $\bc^{\tau}$, 
%the Nash blow-up of $\Delta$ is nothing but the strict transform of $\Delta$ in the blow-up of $\bc^{\tau}$ with center $0$. 
Let $\delta\in \co_{S}$ be a generator of the defining ideal of $\Delta\subset S$. Let $t_{1},\cdots, t_{\tau}$ be a coordinate at $0$ of $S$, and let $[\xi_{1}:\cdots:\xi_{\tau}]$ be the associated coordinate for $\check{\bp}^{\tau-1}$, then $\widetilde{\Delta}$ can be described as
$$
\widetilde{\Delta}=\Big\{\big(s, [\xi_{1}:\cdots:\xi_{\tau}]\big)\in \bc^{\tau}\times \check{\bp}^{\tau-1}\,\big|\,\delta(s)=0, \, \xi_{i}\textstyle{\frac{\partial \delta}{\partial{t_{j}}}}=\xi_{j}\textstyle{\frac{\partial \delta}{\partial{t_{i}}}}, i,j=1,\cdots, \tau\Big\}.
$$
Without loss of generality, we can take the coordinate such that the hyperplane $H$ is defined by $t_{\tau}=0$. Let $\eta_{i}=\xi_{i}/\xi_{\tau}, i=1,\cdots, \tau-1$. Then
$$
\co_{\widetilde{\Delta}, \tilde{s}}=\bc[t_{1},\cdots, t_{\tau}, \eta_{1},\cdots, \eta_{\tau-1}]_{(s, 0)}\Big/\Big(\delta, \textstyle{\frac{\partial \delta}{\partial{t_{i}}}}-\eta_{i}\textstyle{\frac{\partial \delta}{\partial{t_{\tau}}}}, i=1,\cdots, \tau-1\Big).
$$
On the other hand, with the inclusion $\widetilde{\Sigma}\subset X\times \check{\bp}^{\tau-1}$, we obtain that $\km_{\widetilde{\Sigma}, \tilde{x}}$ is generated by the restriction of $\km_{X,x}$ to $\Sigma$ and the functions $\eta_{1},\cdots, \eta_{\tau-1}$. 

%For any element $t\in \km_{S,s}-\km_{S, s}^{2}$, we claim that $\km_{X,x}\subset \psi^{*}\km_{S,s}+\Theta_{X,x}(\psi^{*}(t))$.
By lemma \ref{inf stable}, we have $\psi^{*}(\Theta_{S,s})=\psi^{-1}(\Theta_{S,s})+t\psi(\Theta_{X,x})$. Applying the vector fields on both sides at any element $t\in \km_{S,s}-\km_{S, s}^{2}$, we obtain that 
$$
\co_{X,x}=\psi^{-1}(\co_{S,s})+\Theta_{X,x}(\psi^{*}(t)),
$$
from which we deduce that 
\begin{equation}\label{control mx}
\km_{X,x}= \km_{X,x}\cdot \big[\psi^{-1}(\co_{S,s})+\Theta_{X,x}(\psi^{*}(t))\big] \subset \psi^{*}\km_{S,s}+\Theta_{X,x}(\psi^{*}(t)).
\end{equation}

We look at the ideal $\Theta_{X,x}(\psi^{*}(t))\cdot \co_{\widetilde{\Sigma},\tilde{x}}$ for $t=t_{i}, i=1,\cdots, \tau$.
For any vector field $v\in \Theta_{X,x}$, we have
$$
v(\psi^{*}(\delta))=\sum_{i=1}^{\tau}v\big(\psi^{*}(t_{i})\big)\psi^{*}\big(\textstyle{\frac{\partial\delta}{\partial t_{i}}}\big).
$$
As an element in $\co_{\widetilde{\Sigma},\tilde{x}}$, it is equal to 
\begin{align}
v(\psi^{*}(\delta))&=v\big(\psi^{*}(t_{\tau})\big)\psi^{*}\big({\textstyle{\frac{\partial\delta}{\partial t_{\tau}}}}\big)+\sum_{i=1}^{\tau-1}v\big(\psi^{*}(t_{i})\big) \cdot \eta_{i}\psi^{*}\big(\textstyle{\frac{\partial\delta}{\partial t_{\tau}}}\big) \nonumber
\\
&=\Big[v\big(\psi^{*}(t_{\tau})\big)+\sum_{i=1}^{\tau-1}v\big(\psi^{*}(t_{i})\big) \cdot \eta_{i}\Big] \psi^{*}\big({\textstyle{\frac{\partial\delta}{\partial t_{\tau}}}}\big), \label{v delta}
\end{align}
due to the relation $\textstyle{\frac{\partial \delta}{\partial{t_{i}}}}=\eta_{i}\textstyle{\frac{\partial \delta}{\partial{t_{\tau}}}}$ in $\co_{\widetilde{\Delta},\tilde{s}}$.
One verifies that $v(\psi^{*}(\delta))|_{\Sigma}=0$, i.e. $\Theta_{X}(\psi^{*}(\delta))|_{\Sigma}=0$\footnote{In fact, one can show that ${\rm Im}\{\rd\psi: \Theta_{X}\to \psi^{*}\Theta_{S}\}=\psi^{*}(\Theta_{S}(-\log \Delta))$, where $\Theta_{S}(-\log \Delta)$ is the sheaf of holomorphic vector fields preserving $\Delta$ (cf. \cite{looij icis}, Lem. 6.18).}. Indeed, as $\Sigma$ is normal, it is enough to verify it on the dense open subscheme $\Sigma_{\tau-1}$. It is known that the singularity of the geometric fibers of $\psi$ at points in $\Sigma_{\tau-1}$ is ordinary quadratic. By the openness of versality, it is then enough to verify the assertion for the miniversal deformation of ordinary quadratic singularity, which is clear.  Consequently, as $\psi^{*}\big({\textstyle{\frac{\partial\delta}{\partial t_{\tau}}}}\big)\neq 0$ in $\co_{\widetilde{\Sigma},\tilde{x}}$, we deduce from the equation (\ref{v delta}) the equality in $\co_{\widetilde{\Sigma},\tilde{x}}$:
$$
v\big(\psi^{*}(t_{\tau})\big)=-\sum_{i=1}^{\tau-1}v\big(\psi^{*}(t_{i})\big) \cdot \eta_{i},
$$
which implies that
$$
\Theta_{X,x}(\psi^{*}(t_{\tau}))\subset (\eta_{1},\cdots,\eta_{\tau-1}) \quad \text{in } \co_{\widetilde{\Sigma},\tilde{x}}.
$$
With the equation (\ref{control mx}), we obtain
$$
\km_{X,x}\cdot\co_{\widetilde{\Sigma},\tilde{x}}\subset \psi^{*}\km_{S,s}\cdot\co_{\widetilde{\Sigma},\tilde{x}}+(\eta_{1},\cdots, \eta_{\tau-1})\subset \widetilde{\psi}^{*}(\km_{\widetilde{\Delta}, \tilde{s}}).
$$
As $\km_{\widetilde{\Sigma},\tilde{x}}$ is generated by $\km_{X,x}\cdot\co_{\widetilde{\Sigma},\tilde{x}}$ and $\eta_{1},\cdots, \eta_{\tau-1}$, this implies that $\km_{\widetilde{\Sigma},\tilde{x}}\subset  \widetilde{\psi}^{*}(\km_{\widetilde{\Delta}, \tilde{s}})$ as required.

\end{proof}

\begin{rem}\label{nash nature hypersurface}
In the hypersurface case, one can show that $\rd\psi$ is of rank $\ge \tau-1$. Indeed, let $g_{1}, \cdots, g_{\tau}$ be a set of representative for the Tjurina module $\co_{X_{0}}/Df\cdot \co_{X,0}^{n}$. Without loss of generality, we can take $g_{\tau}=1$. Let 
$$
F(x_{1},\cdots, x_{n},t_{1},\cdots, t_{\tau})=f(x_{1},\cdots, x_{n})+\sum_{i=1}^{\tau-1}t_{i}g_{i}+t_{\tau}.
$$ 
Let $X\subset \bc^{n}\times \bc^{\tau}$ be the closed subscheme defined by $F=0$. The composite $\psi:X\hookrightarrow \bc^{n}\times \bc^{\tau}\to \bc^{\tau}$ is then a miniversal deformation of $X_{0}$. Note that $\partial F/\partial t_{\tau}=1$, the germ $(X,0)$ then admits a coordinate $\bc^{n+\tau-1}\to (X,0)$ defined by
$$
(x_{1}, \cdots, x_{n},t_{1},\cdots, t_{\tau-1})\mapsto \Big(x_{1}, \cdots, x_{n},t_{1},\cdots, t_{\tau-1}, -f(x_{1},\cdots, x_{n})-\sum_{i=1}^{\tau-1}t_{i}g_{i}\Big).
$$
With respect to this coordinate, the morphism $\psi:X\to \bc^{\tau}$ is nothing but the morphism $\bc^{n+\tau-1}\to \bc^{\tau}$ defined by 
$$
(x_{1}, \cdots, x_{n},t_{1},\cdots, t_{\tau-1})\mapsto \Big(t_{1},\cdots, t_{\tau-1}, -f(x_{1},\cdots, x_{n})-\sum_{i=1}^{\tau-1}t_{i}g_{i}\Big)$$
Its Jacobian matrix can be calculated to be
$$
D\psi=\begin{bmatrix}
0&\cdots&0&1&&\\
0&\cdots&0&&\ddots&\\
0&\cdots&0&&&1\\
-\frac{\partial F}{\partial x_{1}}&\cdots & -\frac{\partial F}{\partial x_{n}} & -g_{1}&\cdots& -g_{\tau-1}
\end{bmatrix}_{\tau\times(n+\tau-1)},
$$
which is clearly of rank $\ge \tau-1$. This implies that $\widetilde{\Sigma}=\Sigma$, hence the morphism $\psi|_{\Sigma}:\Sigma\to \Delta$ can be identified with the Nash blow-up of $\Delta$.

\end{rem}

%With theorem \ref{whitney criteria}, we deduce that the sequence of multiplicities $\big(m_{0}(S_{\delta}, y), \cdots, m_{\tau-\delta-1}(S_{\delta}, y)\big)$ are constant for $y\in \Delta_{\mu}$.

\begin{example}

Let $X_{0}$ be the germ of singularity defined by $y^{2}-x^{n}=0, n\ge 2$. Let $X\subset \bc^{2}\times \bc^{n-1}$ be the hypersurface defined by the equation
$$
F(x,y; a_{2},\cdots, a_{n}):=y^{2}-(x^{n}+a_{2}x^{n-2}+\cdots +a_{n-1}x+a_{n})=0.
$$
Then the composite $\psi:X\hookrightarrow \bc^{2}\times \bc^{n-1}\to \bc^{n-1}$ defines a miniversal deformation of $X_{0}$. Notice that $\partial F/\partial a_{n}=1$, whence the germ of $X$ at $0$ admits a coordinate $\bc^{2}\times \bc^{n-2}\mapsto (X,0)$ defined by
$$
(x,y; a_{2},\cdots, a_{n-1})\mapsto \big(x,y, a_{2},\cdots, a_{n-1}, y^{2}-(x^{n}+a_{2}x^{n-2}+\cdots +a_{n-1}x)\big).
$$
With this coordinate, the map $\psi$ becomes $\psi: \bc^{2}\times \bc^{n-2}\mapsto \bc^{n-1}$, which can be described as
$$
(x,y; a_{2},\cdots, a_{n-1})\mapsto \big(a_{2},\cdots, a_{n-1}, y^{2}-(x^{n}+a_{2}x^{n-2}+\cdots +a_{n-1}x)\big).
$$
The induced tangent map is then described by the Jacobian matrix
\begin{equation}\label{jacobian dpsi}
D\psi=\begin{bmatrix}
0&0&1&&\\
0&0&&\ddots&\\
0&0&&&1\\
\frac{\partial F}{\partial x} & 2y & -x^{n-2}&\cdots& -x
\end{bmatrix}_{(n-1)\times n},
\end{equation}
%where $\frac{\partial F}{\partial x}=-\big(nx^{n-1}+(n-2)a_{2}x^{n-3}+\cdots +a_{n-1}\big)$. 
Note that $D\psi$ is of rank at least $n-2$, and that it is singular if and only if
$$
\begin{cases}
\frac{\partial F}{\partial y}=2y=0,& \\
\frac{\partial F}{\partial x}=-\big(nx^{n-1}+(n-2)a_{2}x^{n-3}+\cdots +a_{n-1}\big)=0. 
\end{cases}
$$
These equations then define $\Sigma$ as a closed subscheme of $ \bc^{2}\times \bc^{n-2}$. Note that $\frac{\partial}{\partial a_{n-1}}\big(\frac{\partial F}{\partial x}\big)=-1$,  we obtain that $\Sigma$ admits a coordinate $\bc^{n-2}\to \Sigma$ defined by
$$
(x,a_{2},\cdots, a_{n-2})\mapsto \big(x, 0;a_{2},\cdots, a_{n-2}, -(nx^{n-1}+(n-2)a_{2}x^{n-3}+\cdots +2a_{n-2}x^{1})\big).
$$
Compose with the map $\psi: \bc^{2}\times \bc^{n-2}\to \bc^{n-1}$, we find that the discriminant $\Delta$ can be parametrized by the map $\varphi:\bc^{n-2}\to \Delta$ defined by
$$
(x, a_{2},\cdots, a_{n-2})\mapsto \big(a_{2},\cdots, a_{n-2}, \varphi_{n-1}(x, a_{2},\cdots, a_{n-2}), \varphi_{n}(x, a_{2},\cdots, a_{n-2})\big),
$$
where
\begin{align*}
\varphi_{n-1}(x, a_{2},\cdots, a_{n-2})&= -(nx^{n-1}+(n-2)a_{2}x^{n-3}+\cdots +2a_{n-2}x^{1}),
\\
\varphi_{n}(x, a_{2},\cdots, a_{n-2})&=-\big(x^{n}+a_{2}x^{n-2}+\cdots +a_{n-2}x^{2}+\varphi_{n-1}(x, a_{2},\cdots, a_{n-2})x\big)
\\
&=(n-1)x^{n}+(n-3)a_{2}x^{n-2}+\cdots+a_{n-2}x^{2}.
\end{align*}
Observe with the concrete description (\ref{jacobian dpsi}) of $D\psi$ that ${\rm Im}(D\psi)$ at $\varphi(x,a_{2},\cdots, a_{n-2})$, which is also the tangent space of $\Delta$ at that point in case that the point lies in $\Delta^{\reg}$, equals the hyperplane of $\bc^{n-1}$ generated by the vectors
$$
\begin{bmatrix}
1\\
0\\
\vdots\\
0\\
-x^{n-2}
\end{bmatrix}, 
\begin{bmatrix}
0\\
1\\
\vdots\\
0\\
-x^{n-1}
\end{bmatrix}, 
\cdots\cdots,
\begin{bmatrix}
0\\
\vdots\\
0\\
1\\
-x
\end{bmatrix}, 
$$
which clearly depends only on $x$. In particular, the coordinate $x$ can be recovered from the tangent space. This observation explains why the Nash blow-up of $\Delta$ equals exactly $\Sigma$ as asserted in theorem \ref{nash nature}.

Consider the parametrization $\varphi:\bc^{n-2}\to \Delta\subset \bc^{n-1}$, it has Jacobian matrix
\begin{equation}\label{jacobian dphi}
D\varphi=\begin{bmatrix}
0&1&0&\cdots &0\\
0&0&1&\cdots&0\\
\vdots&\vdots&\vdots&\ddots&\vdots\\
0&0&0&\cdots&1\\
\frac{\partial\varphi_{n-1}}{\partial x}& -(n-2)x^{n-3}&-(n-3)x^{n-4}&\cdots & 2x\\
\frac{\partial\varphi_{n}}{\partial x} & (n-3)x^{n-2}& (n-4)x^{n-3}&\cdots&x^{2}\end{bmatrix}_{(n-1)\times (n-2)}.
\end{equation}
Note that $D\varphi$ is of rank at least $n-3$, and that it is singular if and only if
$$
\begin{cases}
\frac{\partial\varphi_{n-1}}{\partial x}=-\big(n(n-1)x^{n-2}+(n-2)(n-3)a_{2}x^{n-4}+\cdots+2\cdot 1a_{n-2}\big)=0,& \\
\frac{\partial\varphi_{n}}{\partial x}=n(n-1)x^{n-1}+(n-2)(n-3)a_{2}x^{n-3}+\cdots+2\cdot 1 a_{n-2}x=-x\frac{\partial\varphi_{n-1}}{\partial x}=0. 
\end{cases}
$$
Let $\Sigma^{1,1}\subset \Sigma$ be the closed subscheme defined by these two equations, or equivalently by the equation $\frac{\partial\varphi_{n-1}}{\partial x}=0$, it is the critical locus of the morphism $\varphi:\bc^{n-2}\to \Delta$. 
Note that $\frac{\partial}{\partial a_{n-2}}\big(\frac{\partial \varphi_{n-1}}{\partial x}\big)=-2$,  we obtain that $\Sigma^{1,1}$ admits a coordinate $\phi^{1,1}:\bc^{n-3}\to \Sigma^{1,1}\subset \bc^{n-2}$ defined by
$$
(x,a_{2},\cdots, a_{n-3})\mapsto \Big(x,a_{2},\cdots, a_{n-3}, -\frac{1}{2}\big(n(n-1)x^{n-2}+(n-2)(n-3)a_{2}x^{n-4}+\cdots+3\cdot 2a_{n-3}x\big)\Big).
$$
In particular, $\Sigma^{1,1}$ is smooth. It is an example of the \emph{Thom-Boardman strata}.
Let $\Delta^{1,1}$ be the image of $\Sigma^{1,1}$ under the morphism $\psi$,  let $\varphi^{1,1}$ be the composite $\bc^{n-3}\xrightarrow{\phi^{1,1}}\Sigma^{1,1}\to \Delta^{1,1}$, it is the restriction of $\psi$ to $\Sigma^{1,1}$ written with the coordinate $\phi^{1,1}$.
As before, with the concrete description (\ref{jacobian dphi}) of $D\varphi$, we obtain that ${\rm Im}(D\varphi)$ at $\phi(x,a_{2},\cdots, a_{n-3})$ is the linear space of $\bc^{n-1}$ generated by the vectors
$$
\begin{bmatrix}
1\\
0\\
\vdots\\
0\\
-(n-2)x^{n-3}\\
(n-3)x^{n-2}
\end{bmatrix}, 
\begin{bmatrix}
0\\
1\\
\vdots\\
0\\
-(n-3)x^{n-4}\\
(n-4)x^{n-3}
\end{bmatrix}, 
\cdots\cdots,
\begin{bmatrix}
0\\
\vdots\\
0\\
1\\
2x\\
x^{2}
\end{bmatrix}, 
$$
which depends only on $x$. 
Moreover, in case that $\phi(x,a_{2},\cdots, a_{n-3})$ is a smooth point of $\Delta^{1,1}$, ${\rm Im}(D\varphi)$ coincides with the tangent space of $\Delta^{1,1}$ at this point. Again, this implies that the Nash blow-up of $\Delta^{1,1}$ is exactly $\Sigma^{1,1}$ as the coordinate $x$ can be recovered from the tangent space.

We can proceed to consider the critical locus $\Sigma^{1,1,1}$ of the morphism $\varphi^{1,1}:\Sigma^{1,1}\to \Delta^{1,1}$ and this process can be iterated, we get a sequence of smooth manifolds
$\Sigma\supsetneq\Sigma^{1,1}\supsetneq \Sigma^{1,1, 1}\supsetneq \cdots $, called \emph{Thom-Boardman strata}.  
Let $\Delta^{1,\cdots,1}$ be the image of $\Sigma^{1,\cdots, 1}$ under the morphism $\psi$. With the same reasoning as before, we can show that $\Sigma^{1,\cdots, 1}$ is the Nash blow-up of $\Delta^{1,\cdots,1}$.

Despite this beautiful property of the Thom-Boardman strata, they are not sufficient to understand the geometry of the discriminant $\Delta$. For example, they miss completely the strict $\delta$-strata, none of them appears as $\Delta^{1,\cdots,1}$.

\cqfd

\end{example}

\subsection{Whitney regularity}\label{whitney regularity}

Whitney stratification is the most widely accepted notion of “equisingularity''. In this section, we make a brief recall of Thom-Mather's theory of Whitney stratification and Teissier's numerical criteria for the Whitney regularity condition. 
They are used to show the Whitney regularity of the union of the strict $\delta$-strata over the equisingular strata. We then use this result to show that the sheaves $\ch^{*}(j_{!*}\cf^{i})$ are locally constant over the equisingular strata.

\paragraph{(1) Whitney stratification}

Let $X$ be a reduced complex analytic variety, a \emph{stratification} of $X$ is a decomposition into disjoint union $X=\bigsqcup_{\alpha\in A}X_{\alpha}$ of locally closed smooth complex analytic subvarieties, called \emph{strata}, satisfying the \emph{frontier condition}: 
$$X_{\alpha}\cap \overline{X_{\beta}}\neq \emptyset 
\iff 
X_{\alpha}\subset \overline{X_{\beta}},\quad \forall \, \alpha,\beta\in A.
$$
The stratification is said to be \emph{locally finite} if the number of strata is locally finite.

\begin{defn}

Let $X$ be a reduced locally closed subvariety of a smooth complex analytic variety $M$, let $\{X_{\alpha}\}_{\alpha\in A}$ be a stratification of $X$. A pair of strata $(X_{\alpha}, X_{\beta})$ with $X_{\alpha}\subset \overline{X_{\beta}}$ is said to satisfy Whitney's condition at $y\in X_{\alpha}$ if

\begin{enumerate}[topsep=1pt, itemsep=1pt, label=$(\alph*)$]

\item
For any sequence $\{x_{i}\}_{i\in \bn}\subset X_{\beta}$ tending to $y$, we have
$$
T_{y}X_{\alpha}\subset \lim_{i\to \infty} T_{x_{i}}X_{\beta}.
$$

\item
For any sequences $\{y_{i}\}_{i\in \bn}\subset X_{\alpha}, \{x_{i}\}_{i\in \bn}\subset X_{\beta}$ tending to $y$, we have
$$
\lim_{i\to \infty} \overline{x_{i}y_{i}} \subset \lim_{i\to \infty} T_{x_{i}}X_{\beta},
$$
where $\overline{x_{i}y_{i}}$ is the secant line passing through $x_{i}$ and $y_{i}$ for some local coordinate system of $M$. 

\end{enumerate}

\noindent In case that the conditions are satisfied for all points of $X_{\alpha}$, we say that $X_{\beta}$ is \emph{Whitney regular} over $X_{\alpha}$. 
If the Whitney conditions are satisfied for any pair of strata $(X_{\alpha}, X_{\beta})$ with $X_{\alpha}\subset \overline{X_{\beta}}$, we call the stratification $\{X_{\alpha}\}_{\alpha\in A}$ a \emph{Whitney stratification}.
\end{defn}

The above conditions will be called Whitney's condition $(a)$ and $(b)$. It is known that the condition $(b)$ implies condition $(a)$.

\begin{thm}[Whitney \cite{whitney}]

Let $X$ be a reduced locally closed subvariety of a smooth complex analytic variety $M$, then there exists a locally finite Whitney stratification of $M$ such that $X$ is a union of strata. 

\end{thm}

The Whitney stratified spaces enjoy a nice locally topologically trivial property along the strata:

\begin{thm}[Thom-Mather \cite{mather notes}, Consequence of Thom's first isotopy lemma]\label{thom-mather}

Let $X$ be a reduced locally closed subvariety of a smooth complex analytic variety $M$. Let $\{X_{\alpha}\}_{\alpha\in A}$ be a Whitney stratification of $X$. For any point $x$, let $X_{\alpha}$ be the stratum containing $x$. Then for any local embedding $(X, x)\subset (\bc^{n},0)$ induced from a local coordinate of $M$, and any local  retraction $\rho: (\bc^{n},0)\to (X_{\alpha},0)$, there exists $\ep_{0}>0$ such that for any $0<\ep<\ep_{0}$ there exists $\eta_{\ep}>0$ such that for any $0<\eta<\eta_{\ep}$ there exists a homeomorphism $h$ making the diagram commutative
$$
\begin{tikzcd}
B_{0}(\ep)\cap \rho^{-1}(X_{\alpha}\cap B_{0}(\eta))\arrow[rr,"h"] \arrow[dr, "\rho"] & & (B_{0}(\ep)\cap \rho^{-1}(x)) \times (X_{\alpha}\cap B_{0}(\eta))\arrow[dl, "\pr_{2}"]
\\
& X_{\alpha}\cap B_{0}(\eta) &
\end{tikzcd},
$$ 
which induces for any stratum $X_{\beta}$ containing $X_{\alpha}$ at its frontier a homeomorphism
$$
\overline{X_{\beta}}\cap B_{0}(\ep)\cap \rho^{-1}(X_{\alpha}\cap B_{0}(\eta))
\xrightarrow{h} 
\big(\overline{X_{\beta}}\cap B_{0}(\ep)\cap \rho^{-1}(x)\big) \times (X_{\alpha}\cap B_{0}(\eta)),
$$
where $B_{0}(r)$ is the open ball in $\bc^{n}$ with center $0$ and radius $r$.

\end{thm}

Conversely, the Whitney stratification can be characterized topologically as follows: Let $\{X_{\alpha}\}_{\alpha\in A} $ be a stratification of a reduced complex analytic variety $X$, let $(X_{\alpha}, X_{\beta})$ be any pair satisfying $X_{\alpha}\subset \overline{X_{\beta}}$. The pair is said to satisfy the condition $(TT)$ (local topological triviality) if it satisfies the conclusion of Thom-Mather's theorem. It is said to satisfy the condition $(TT^{*})$ if for any $x\in X_{\alpha}$, any local embedding $(X,x)\subset (\bc^{n}, 0)$ and any $i=\dim(X_{\alpha}), \cdots, n$, there exists a dense open subscheme $U_{\alpha}$ of the Grassmannian of linear subspaces of $\bc^{n}$ of dimension $i$ containing $T_{x}(X_{\alpha})$, such that for any submanifold $H\subset \bc^{n}$ containing $(X_{\alpha}, x)$ with $T_{x}(H)\in U_{\alpha}$, the slice $(X_{\alpha}\cap H, X_{\beta}\cap H)$ satisfies the condition $(TT)$.

\begin{thm}[L\^e-Teissier \cite{le-teissier}]

Let $\{X_{\alpha}\}_{\alpha\in A}$ be a stratification of a complex analytic variety $X$, then it is a  Whitney stratification if and only if any pair $X_{\alpha}\subset  \overline{X_{\beta}}$ of it satisfies the condition $(TT^{*})$.

\end{thm}

\paragraph{(2) Teissier's criteria}

Teissier \cite{teissier polar 2}  has found a beautiful characterization of the Whitney regularity  condition, we make a brief recall of his result.

%To begin with, we reformulate the Whitney condition in a more intrinsic way. 
Let $X$ be a reduced complex analytic variety of pure dimension $d$, let $Y$ be a smooth irreducible closed  subvariety of $X^{\rm sing}$ of dimension $d'$.  For $y\in Y$, without loss of generality, we can assume the existence of a local embedding $(X,y)\subset (\bc^{n},0)$ such that $Y$ is identified with the subspace $\bc^{d'}\subset \bc^{n}$.
Henceforth we denote by $(X,y)$ the germ at $y$ of $X$ equipped with this local embedding. 
The Whitney condition involves taking limit of secant lines and tangent spaces, this led to the consideration of the blow-up $\bl_{Y}(X)$ of $X$ with center $Y$ and the {Nash blow-up} $N(X)$ of $X$, which is the Zariski closure in $X\times \gr_{d,n}$ of the graph of the Gauss map
$$
X^{\reg}\to  \gr_{d,n},\quad x\mapsto T_{x}(X^{\reg}),
$$ 
where $\gr_{d,n}$ is the Grassmannian of $d$-dimensional linear subspace of $\bc^{n}$, and the tangent space $T_{x}(X^{\reg})$ is identified with the $d$-dimensional linear subspace of $\bc^{n}$ parallel to it. 
For technical reasons, we consider the \emph{projective conormal space} $C(X)$ of $X$ instead of the Nash blow-up. 
Recall that the \emph{conormal bundle} $T_{X^{\reg}}^{*}\bc^{n}$  of $X^{\reg}$ in the ambient space $\bc^{n}$ is defined as
$$
T_{X^{\reg}}^{*}\bc^{n}=\big\{(x,\xi)\in T^{*}\bc^{n}\mid x\in X, \xi\in T_{x}^{*}(\bc^{n})\text{ such that } \pair{v,\xi}=0 \text{ for all }v\in T_{x}
(X^{\reg})\big\},
$$
i.e. at each point $x\in X^{\reg}$, it parametrizes the hyperplanes in $\bc^{n}$ containing the tangent space $T_{x}(X^{\reg})$. Let $T_{X}^{*}\bc^{n}$ be the Zariski closure of $T^{*}_{X^{\reg}}\bc^{n}$ in $T^{*}\bc^{n}$, and let $C(X)$ be the projectivization of $T^{*}_{X}(\bc^{n})$. It is clear that $C(X)$ is of dimension $n-1$. 
By construction, we get the diagram
$$
\begin{tikzcd}
C(X)\arrow[d, "\kappa"] \arrow[r, hook] \arrow[dr,"\lambda"]& X\times \check{\bp}^{n-1}\arrow[d,"\pr_{2}" ]
\\
X&\check{\bp}^{n-1}
\end{tikzcd}.
$$
Let $e_{Y}: \bl_{Y}(X)\to X$ be the blow-up of $X$ with center $Y$, let $\hat{e}_{Y}: E_{Y}C(X)\to C(X)$ be the blow-up of $C(X)$ with center $\kappa^{-1}(Y)$, we obtain then the commutative diagram
\begin{equation}\label{normal-conormal}
\begin{tikzcd}
E_{Y}C(X)\arrow[r, "\hat{e}_{Y}"]\arrow[dd, "\kappa'"] \arrow[ddr, "\xi"]& C(X)\arrow[dd, "\kappa"] \arrow[r, hook] \arrow[dr,"\lambda"]& X\times \check{\bp}^{n-1}\arrow[d,"\pr_{2}" ]
\\
&&\check{\bp}^{n-1}
\\
\bl_{Y}(X)\arrow[r,"{e}_{Y}"] & X&
\end{tikzcd},
\end{equation}
where the existence of the morphism $\kappa'$ is due to the universal property of the blow-up. 
Note that $E_{Y}C(X)$ coincides with the Zariski closure of the inverse image of $\bl_{Y}(X)-e_{Y}^{-1}(Y)$ in $\bl_{Y}(X)\times_{X}C(X)$.
%By construction, we get natural inclusions
%$$\bl_{Y}(X)\hookrightarrow X\times \bp^{n-d'-1} \quad\text{and}\quad E_{Y}C(X)\hookrightarrow \bl_{Y}(X)\times_{X}C(X)\hookrightarrow X\times {\bp}^{n-d'-1}\times \check{\bp}^{n-1}.
%$$
By construction, the pre-image $\xi^{-1}(Y)$ and the morphisms
$$
\begin{tikzcd}[column sep={6em,between origins}]
&\xi^{-1}(Y)\arrow[dl, "\kappa'"'] \arrow[dr, "\hat{e}_{Y}"]&\\
\bl_{Y}(X)& &C(X)
\end{tikzcd}
$$
encodes the information about the limit secant lines and limit tangent spaces. For any $y\in Y$, the above diagram induces morphisms
\begin{equation}\label{fiber of xi}
\begin{tikzcd}[column sep={6em,between origins}]
&\xi^{-1}(y)\arrow[dl, "q_{1}"'] \arrow[dr, "q_{2}"]&\\
\bp^{n-d'-1}& &\check{\bp}^{n-1}
\end{tikzcd}.
\end{equation}
To analyse the fiber $\xi^{-1}(y)$ via the above diagram, we are led to the local polar varieties: 

\begin{defn}

Let $H_{d-k+1}$ be a linear subspace of $\bc^{n}$ of codimension $d-k+1$, let $\pi:\bc^{n}\to \bc^{d-k+1}$ be the linear projection with kernel $H_{d-k+1}$. The Zariski closure in $(X,y)$ of the critical locus for the restriction $\pi:(X^{\reg},y)\subset \bc^{n}\to \bc^{d-k+1}$ is called the \emph{local polar variety} of $X$ at $y$ with respect to $H_{d-k+1}$, denoted $P_{y}(X, H_{d-k+1})$.

\end{defn}

Let $\check{H}^{d-k}\subset \check{\bp}^{n-1}$ be the dual of $H_{d-k+1}\subset \bc^{n}$, it consists of the hyperplanes in $\bc^{n}$ containing $H_{d-k+1}$, hence is a linear projective subspace of dimension $d-k$.  

\begin{prop}[Teissier \cite{teissier polar 2}, Chap. IV, corollaire 1.3.2]\label{teissier local polar}

For a generic choice of $H_{d-k+1}$, the local polar variety $P_{y}(X, H_{d-k+1})$ coincides with the image $\kappa(\lambda^{-1}(\check{H}^{d-k}))$ in $(X,y)$. It is a reduced closed subvariety of $(X,y)$, either of pure codimension $k$ or empty.

\end{prop}

%The proof relies on the fact that the {Schubert variety} of $\gr_{d,n}$ defined by the condition
%$$
%\big\{H\in \gr_{d,n}\mid \dim(H\cap H_{d-k+1})\ge k\big\},
%$$
%is an irreducible closed subvariety of $\gr_{d,n}$ of codimension $k$. 
The coincidence of  $P_{y}(X, H_{d-k+1})$ with $\kappa(\lambda^{-1}(\check{H}^{d-k}))$ implies that 
$q_{1}(q_{2}^{-1}(\check{H}^{d-k}))$
coincides with the projectivization of the normal cone
%$
%C_{y}(P_{y}(X, H_{d-k+1}))
%$
of $P_{y}(X, H_{d-k+1})$ at $y$, where $q_{1},q_{2}$ refers to the morphisms in the diagram (\ref{fiber of xi}). This explains the relationship between the local polar varieties and the fiber $\xi^{-1}(y)$.

\begin{prop}[Teissier \cite{teissier polar 2}, Chap. IV, th\'eor\`eme 3.1]

For $k=0,\cdots, d-1$, let $H_{d-k+1}$ be a generic linear subspace of $\bc^{n}$ of codimension $d-k+1$, then the multiplicity of $P_{y}(X,H_{d-k+1})$ at $y$ depends only on the analytic local algebra $\co_{X,y}$.

\end{prop}

Hence the multiplicity of $P_{y}(X,H_{d-k+1})$ at $y$ is independent of the choice of $H_{d-k+1}$, we will denote it by $m_{k}(X,y)$. 
For $k=0$, the restriction $\pi:X^{\reg}\to \bc^{d+1}$ is a local embedding at a neighbourhood of $y$, hence the local polar variety $P_{y}(X, H_{d+1})$ is exactly the germ of $X$ at $y$. In particular, $m_{0}(X,y)$ coincides with the multiplicity of $X$ at $y$.

\begin{thm}[Teissier \cite{teissier polar 2}, Chap. V, th\'eor\`eme 1.2]\label{whitney criteria}

Let $(X,y)\subset (\bc^{n},0)$ be a local embedding of reduced complex analytic variety, let $Y$ be a smooth complex analytic subvariety of $X^{\rm sing}$ containing $y$,  the following conditions are equivalent:
\begin{enumerate}[topsep=2pt, noitemsep, label=$(\arabic*)$]

\item

The sequence of multiplicities $\big(m_{0}(X,y),\cdots,m_{d-1}(X,y)\big)$ is locally constant at  $y$,

\item

$\dim(\xi^{-1}(y))=n-2-\dim(Y)$.

\item
The pair $(X^{\reg}, Y)$ satisfies Whitney's condition at $y$,

\end{enumerate}

\end{thm}

Geometrically, the multiplicity of a germ of complex analytic subvariety $(V,0)\subset (\bc^{n},0)$ equals the intersection number of the germ with a generic plane in $\bc^{n}$ of complementary dimension and sufficiently close to $0$. For the local polar varieties, we have a convenient choice of such planes:

\begin{prop}[Henry-Merle \cite{hm}]\label{hm transversal}

For $k=0,1,\cdots, d-1$, let $H_{d-k+1}$ be the generic plane defining the local polar variety $P_{y}(X, H_{d-k+1})$, let $H_{d-k}$ be a generic plane of $\bc^{n}$ of codimension $d-k$ containing $H_{d-k+1}$, then $H_{d-k}$ intersects $P_{y}(X, H_{d-k+1})$ transversally at $y$.

\end{prop}

In particular, a sufficiently small translation of the plane $H_{d-k}$ intersects $P_{y}(X, H_{d-k+1})$ transversally. Note that they are of complementary dimensions, hence the multiplicity $m_{k}(X,y)$ equals their intersection number.

\paragraph{(3) Whitney regularity along the equisingular strata}

Recall that two germs of plane curve singularities $(C^{\circ}_{1}, 0), (C^{\circ}_{2},0)$ are said to be \emph{of the same topological type} if there exist representatives $C_{i}\subset U_{i}$ of $C^{\circ}_{i}$ inside open neighbourhood $U_{i}\subset \bc^{2}$ of $0$, $i=1,2$, such that there exists homeomorphism $U_{1}\sim U_{2}$ sending $C_{1}$ to $C_{2}$.

\begin{thm}[Zariski \cite{zariski equisingularity}, Teissier \cite{teissier resolution}]\label{equisingularity}

Let $S$ be a germ of smooth complex analytic variety, let $(C, 0)\subset (S\times \bc^{2}, 0)$ be a germ of reduced plane curves over $S$, let $S'$ be $S\times \{0\}$. Suppose that the family $C\to S$ is smooth outside the section $S'$. 
The following conditions are equivalent:
\begin{enumerate}[topsep=2pt, noitemsep, label=\textup{(\arabic*)}]
\item

All the fibers $C_{s}$ have the same topological type.

\item

All the fibers $C_{s}$ have the same Milnor number.

\item

The $\delta$-invariant and the number of branches of $C_{s}$ are independent of $s\in S$.

\item

For any two points $s_{1},s_{2}$ of $S$, there exists a bijection between the set of branches of $C_{s_{1}}$ and $C_{s_{2}}$ which preserve the Puiseux characteristics of the branches and the intersection number between any two distinct branches.

\item

Let $\phi:\widetilde{C}\to C$ be the normalization of $C$. The composite $\widetilde{C}\xrightarrow{\phi} C\to S$ is  submersive at a neighbourhood of $\phi^{-1}(S')$, which induces a simultaneous resolution of singularity for all the fibers $C_{s}$, such that the restriction $\phi:\phi^{-1}(S')_{\rm red}\to S'$ is a trivial covering.  

\item

The pair $(C-S', S')$ satisfies Whitney's conditions.

\end{enumerate}

\end{thm}

The family $C\to S$ will be called an \emph{equisingular deformation} if it satisfies any of the  conditions in the theorem.

\begin{thm}[Teissier \cite{zariski plane branch}, Appendix, \S 3.1]

Let $\psi: C\to B$ be a miniversal deformation of a plane curve singularity $C_{0}$, let $\Delta_{\mu}\subset B$ be the closed subscheme which parametrizes all the equisingular deformations of $C_{0}$, then $\Delta_{\mu}$ is smooth over $\bc$. Moreover, let $\Sigma$ be the critical locus of $\psi$ and $\Sigma_{\mu}=(\psi|_{\Sigma})^{-1}(\Delta_{\mu})_{\rm red}$, then the restriction $\psi:\Sigma_{\mu}\to \Delta_{\mu}$ is an isomorphism.  

\end{thm}

The theorem applies to our family $\pi:\cc\to \cb$.
Let $\Delta_{\mu}\subset \Delta$ be the closed subscheme parametrizing equisingular deformations of $C_{\gamma}$, and let $\Sigma_{\mu}:=(\pi|_{\Sigma})^{-1}(\Delta_{\mu})_{\rm red}$. 
By construction, $\Delta_{\mu}$ is contained in the $\delta$-stratum $\cb_{\delta_{\gamma}}$, hence it is contained in the closure of all the strict $\delta$-strata $\cb_{\delta}^{\circ}$.

\begin{thm}\label{key regularity}

The union of all the strict $\delta$-strata $\cb_{\delta}^{\circ}$ is Whitney regular over $\Delta_{\mu}$.

\end{thm}

\begin{proof}

The stratum $\cb_{0}^{\circ}$ is nothing but $\cb-\Delta$. As both $\cb$ and $\Delta_{\mu}$ are smooth over $\bc$, the stratum $\cb_{0}^{\circ}$ is Whitney regular over $\Delta_{\mu}$.

The stratum $\cb_{1}^{\circ}$ is just the smooth locus $\Delta^{\reg}$. By theorem \ref{nash nature} and remark \ref{nash nature hypersurface}, the Nash blow-up of $\Delta$, which is also the projective conormal space $C(\Delta)$, is smooth over $\bc$, and it can be identified with the critical locus $\Sigma$.  As both $\Sigma$ and $\Sigma_{\mu}$ are nonsingular, the condition $(2)$ in theorem \ref{whitney criteria} is satisfied for the pair $(\Delta, \Delta_{\mu})$, whence the Whitney regularity of $\cb_{1}^{\circ}$ over $\Delta_{\mu}$.

For the other strata $\cb_{\delta}^{\circ}, \delta=2,\cdots, \delta_{\gamma}$, note that they appear as self-intersections of $\cb_{1}^{\circ}$, and that the intersections are all transversal.  With the above result for $\cb_{1}^{\circ}$, any local component of $\cb_{1}^{\circ}$ passing through $\cb_{\delta}^{\circ}$ is Whitney regular over $\Delta_{\mu}$. 
This implies the Whitney conditions for the pair $(\cb_{\delta}^{\circ}, \Delta_{\mu})$. 
%as the tangent space of $\cb_{\delta}^{\circ}$ at any point on it is the intersection of that of the local components of $\cb_{1}^{\circ}$ at this point.
Indeed, for any point $y\in \Delta_{\mu}$, let $U$ be any small neighbourhood of $y$ in $\cb$ which intersects non-trivially with $\Delta_{\mu}$ and with all the local components of $\cb_{1}^{\circ}$ passing through $\cb_{\delta}^{\circ}$. Let $V_{1},\cdots, V_{\delta}$ be the intersection of $U$ with these local components. Then for any sequences $\{y_{i}\}_{i\in \bn}\subset \Delta_{\mu}\cap U, \{x_{i}\}_{i\in \bn}\subset \cb_{\delta}^{\circ}\cap U$ tending to $y\in \Delta_{\mu}$, we have Whitney's condition $(a)$
$$
\lim_{i\to \infty}T_{y_{i}}(\Delta_{\mu}) \subset \bigcap_{i=1}^{\delta} \lim_{i\to \infty} T_{x_{i}}V_{i}=\lim_{i\to \infty} T_{x_{i}}\cb_{\delta}^{\circ},
$$
and condition $(b)$
$$
\lim_{i\to \infty} \overline{x_{i}y_{i}} \subset \bigcap_{i=1}^{\delta} \lim_{i\to \infty} T_{x_{i}}V_{i}=\lim_{i\to \infty} T_{x_{i}}\cb_{\delta}^{\circ}.
$$

\end{proof}

\paragraph{(4) Proof of the main theorems}

As we have explained in remark \ref{reduce main}, theorem \ref{main CJ} implies theorem \ref{main}. Hence it is enough to prove the former. 

\begin{lem}

Let $\gamma, \gamma'\in \kg[\![\varep]\!]$ be semisimple regular element, if $\gamma$ and $\gamma'$ have the same root valuation datum, then $C_{\gamma}$ and $C_{\gamma'}$ are equisingular.

\end{lem}

\begin{proof}

By theorem \ref{equisingularity}, $\spf(\co[\gamma])$ and $\spf(\co[\gamma'])$ are equisingular if and only if there is a bijection between the branches of the singularities such that the datum consisting of the Puiseux characteristic of the branches and the intersection numbers between pairs of branches are the same under the bijection.  We will show that the root valuation datum of $\gamma$ determines such datum.

Up to conjugation, we can write $\gamma=\diag(\gamma_{i})_{i=1}^{r}$ with $\gamma_{i}\in \mathfrak{gl}_{d_{i}}(\co)$ such that the characteristic polynomial of $\gamma_{i}$ is irreducible over $F$. We have then the decomposition into irreducible components $\spf(\co[\gamma])=\bigcup_{i=1}^{r}\spf(\co[\gamma_{i}])$.
For each $i$, the eigenvalues of $\gamma_{i}$ will be the Galois conjugates of an element of the form
$\sum_{n=n_{i}}^{\infty}a_{n}\varep^{n/d_{i}}$
with $a_{n}\in \bc$ and $a_{n_{i}}\neq 0$ for some $n_{i}\in \bn_{0}$.

In case that $n_{i}\ge d_{i}$, the Puiseux parametrization of the branch of $\spf(\co[\gamma_{i}])$ is
\begin{equation}\label{puiseux}
\begin{cases}
x=t^{d_{i}},&\\
y=\sum_{n=n_{i}}^{\infty}a_{n}t^{n}.
\end{cases}
\end{equation}
By definition, the Puiseux characteristic is defined inductively as: Let $\beta_{0}=d_{i}$. If $\beta_{0}=1$, the Puiseux characteristic is just $(\beta_{0})=(1)$. If $\beta_{0}>1$, let $\beta_{1}$ be the smallest integer $n$ such that $a_{n}\neq 0$ and $d_{i}\nmid n$. Let $e_{1}=(\beta_{0}, \beta_{1})$. If $e_{1}=1$, the Puiseux characteristic is $(\beta_{0};\beta_{1})$, otherwise let $\beta_{2}$ be the smallest integer $n$ such that $a_{n}\neq 0$ and $e_{1}\nmid \beta_{2}$. In general, if $\beta_{\nu}$ is defined, let $e_{\nu}=(\beta_{0}, \cdots, \beta_{\nu})$. If $e_{\nu}=1$, the Puiseux characteristic is the $(\nu+1)$-tuple $(\beta_{0};\beta_{1},\cdots, \beta_{\nu})$, otherwise let $\beta_{\nu+1}$ be the smallest integer $n$ such that $a_{n}\neq 0$ and $e_{\nu}\nmid \beta_{\nu+1}$. 
This process terminates as $e_{0}>e_{1}>\cdots$, and we get the Puiseux characteristic $(\beta_{0};\beta_{1},\cdots, \beta_{g})$. Let $m_{\nu}=\beta_{\nu}/e_{\nu}$ and $n_{\nu}=e_{\nu-1}/e_{\nu}$ with $e_{0}:=d_{i}$, $\nu=1,\cdots,g$. The sequence of pairs $(m_{1}, n_{1}),\cdots, (m_{g},n_{g})$ are called the \emph{characteristic pairs} of the parametrization (\ref{puiseux}).
Note that we can reconstruct $(\beta_{0};\beta_{1},\cdots, \beta_{g})$ from the characteristic pairs $(m_{1}, n_{1}),\cdots, (m_{g},n_{g})$, hence the two datum are equivalent.  
Indeed, as $e_{g}=1$, we have $\beta_{g}=m_{g}$ and $e_{\nu-1}=n_{\nu}n_{\nu+1}\cdots n_{g}$ and so $\beta_{\nu-1}=m_{\nu-1}\cdot n_{\nu}n_{\nu+1}\cdots n_{g}$ for $\nu=1,\cdots, g$.

Note that $d_{i}=e_{0}=n_{1}\cdots n_{\nu}\cdot e_{\nu}$ and $\beta_{\nu}=e_{\nu}m_{\nu}$, and so $\beta_{\nu}/d_{i}=m_{\nu}/n_{1}\cdots n_{\nu}$. Moreover, $(m_{\nu}, n_{\nu})=(\beta_{\nu}, e_{\nu-1})/e_{\nu}=1$ and $m_{\nu}/n_{1}\cdots n_{\nu}<m_{\nu+1}/n_{1}\cdots n_{\nu+1}$ by construction. 
In terms of the characteristic pairs, we can rewrite
{\small
\begin{align*}
\sum_{n=n_{i}}^{\infty}a_{n}\varep^{n/d_{i}}=\sum_{i=0}^{k_{0}} a_{0, i}\varep^{i}+\sum_{i=0}^{k_{1}} a_{1,i}\varep^{(m_{1}+i)/n_{1}}+\cdots+\sum_{i=0}^{k_{\nu}} a_{\nu,i}\varep^{(m_{\nu}+i)/n_{1}\cdots n_{\nu}}+\cdots+\sum_{i=0}^{+\infty} a_{g,i}\varep^{(m_{g}+i)/n_{1}\cdots n_{g}},
\end{align*}
}

\noindent with $a_{j,i}\in \bc$. As the generator $\tau\in \gal(\bc(\!(\varep^{1/d_{i}})\!)/\bc(\!(\varep)\!))$ sends $\varep^{1/d_{i}}$ to $\zeta_{d_{i}}\varep^{1/d_{i}}$, we obtain that $m_{\nu}/n_{1}\cdots n_{\nu}$ appears as one of the root valuations of $\gamma_{i}$ for all $\nu$, and they are all the root valuations. 
As $(m_{\nu}, n_{\nu})=1$ for all $\nu$, the characteristic pairs $(m_{1}, n_{1}),\cdots, (m_{g},n_{g})$ can be reconstructed from the sequence $(m_{\nu}/n_{1}\cdots n_{\nu})_{\nu=1}^{g}$. Indeed, we get the pair $(m_{1}, n_{1})$ from $m_{1}/n_{1}$, and $(m_{2},n_{2})$ from $n_{1} \cdot \frac{m_{2}}{n_{1}n_{2}}$, and so on. In this way, the root valuation datum of $\gamma_{i}$ determines the characteristic pairs hence also the Puiseux characteristic of the branch $\spf(\co[\gamma_{i}])$ in this case.

In case that $n_{i}<d_{i}$, we still have the parametrization (\ref{puiseux}) for the branch $\spf(\co[\gamma_{i}])$, but it is not the Puiseux parametrization, to get it we need to inverse the role of $x$ and $y$. 
Let $(m_{1}, n_{1}),\cdots, (m_{g},n_{g})$ be the characteristic pairs associated to the parametrization (\ref{puiseux}) as before, let $(m'_{1}, n'_{1}),\cdots, (m'_{g'},n'_{g'})$ be the characteristic pairs for the Puiseux parametrization of $\spf(\co[\gamma_{i}])$.
According to the inversion formula of Abhyankar \cite{Ab}, they are related as follows: $g'=g$, $m'_{1}=n_{1}$, $n'_{1}=m_{1}$ and 
$$
m_{i}'=m_{i}-(m_{1}-n_{1})n_{2}\cdots n_{i},\quad
n_{i}'=n_{i}, \quad \text{for }i=2,\cdots, g.
$$
As before, the root valuation datum of $\gamma_{i}$ determines the characteristic pairs $(m_{1}, n_{1}),\cdots, (m_{g},n_{g})$. With the inversion formula, it determines the characteristic pairs $(m'_{1}, n'_{1}),\cdots, (m'_{g'},n'_{g'})$ hence also the Puiseux characteristic of $\spf(\co[\gamma_{i}])$. 

In any case, the intersection number between the branches $\spf(\co[\gamma_{i}])$ and $\spf(\co[\gamma_{j}])$ is given by the formula of Halphen-Zeuthen (cf. \cite{greuel}, Chap. I, \S3, prop. 3.10): We write eigenvalues of $\gamma_{i}$ and $\gamma_{j}$ in the form
$
u_{k}(\varep^{1/d_{i}d_{j}})
$
and $v_{l}(\varep^{1/d_{i}d_{j}})$ respectively, with $u_{k}\in \bc[\![\varep]\!]$ for $k=1,\cdots, d_{i}$, $v_{l}\in \bc[\![\varep]\!]$ for $l=1,\cdots, d_{j}$. Then the intersection number equals
$$
\sum_{k=1}^{d_{i}}\sum_{l=1}^{d_{j}} \val\big(u_{k}(\varep^{1/d_{i}d_{j}})-v_{l}(\varep^{1/d_{i}d_{j}})\big),
$$
which is again determined by the root valuation datum of $\gamma$.

\end{proof}

For the proof of theorem \ref{main CJ}, by Ng\^o's support theorem and the above lemma,  it is enough to show that the groups $\ch^{*}(j_{!*}\cf^{i})$ are locally constant over $\Delta_{\mu}$ for all $i$.

For any point $z\in \Delta_{\mu}$, we can make microlocal analysis of the fiber $(j_{!*}\cf^{i})_{z}$ with the same process as explained in \S\ref{first reduction} and \S\ref{iterate}. 
The Morse groups for this inductive process are calculated in the same way with the same results. 
By proposition \ref{morse group}, the Morse groups are non-trivial only over the strict $\delta$-strata $\cb_{\delta}^{\circ}$. 
By theorem \ref{key regularity}, the union of the strict $\delta$-strata $\bigcup_{\delta=0}^{\delta_{\gamma}} \cb_{\delta}^{\circ}$ is Whitney regular over $\Delta_{\mu}$, 
hence by theorem \ref{thom-mather} they are locally topologically trivial along it.  
Hence the process of patching up the Morse groups to calculate $\ch^{*}(j_{!*}\cf^{i})_{z}$ as explained in remark \ref{patch up morse groups}
is everywhere the same on $\Delta_{\mu}$, and so the group $\ch^{*}(j_{!*}\cf^{i})_{z}$ must be locally constant on it. This finishes the proof of theorem \ref{main CJ}.

\section{Appendix: Microlocal analysis of the intermediate extension}\label{iterated milnor}

Let $X$ be a normal complex analytic variety of dimension $d$, let $j:U\to X$ be the inclusion of a dense open subvariety, let $\cf$ be a locally constant sheaf on $U$. 
Following ideas of Goresky and MacPherson \cite{gm morse}, Chap. 1, \S1.5, page 17, we explain an inductive process to calculate the fiber $(j_{!*}\cf)_{x}$ at any point $x\in X-U$, we call it \emph{microlocal analysis}\footnote{According to Goresky and MacPherson, this process should be called \emph{(micro)$^{n}$-local analysis}. } of $(j_{!*}\cf)_{x}$.  We fix an embedding $(X,x)\subset (\bc^{N},0)$.

\paragraph{(1) Reduction to the nearby fiber}

Let $\pi:\bc^{N}\to \bc$ be a generic projection, it induces a morphism $p: X_{\{x\}}\to \ba^{1}_{\{0\}}$. Let $\eta$ be the generic point of $\ba^{1}_{\{0\}}$ and let $\bar{\eta}$ be a geometric point over it. Let $R\Psi_{\bar{\eta}}$ (resp. $R\Phi_{\bar{\eta}}$) be the nearby cycle (resp. vanishing cycle) with respect to the projection $p$ at the point $0$, let $T$ be the monodromy action on $R\Psi_{\bar{\eta}}$. Let $i_{x}:\{x\}\hookrightarrow X$ be the natural inclusion.

\begin{thm}[Goresky-MacPherson \cite{gm morse}, Part II, \S6] \label{ic via nearby}

Under the above setting, we have
$$
\ch^{m}((j_{!*}\cf[d])_{{x}})=\begin{cases}
{\rm Ker}\Big\{i_{x}^{*}R^{-1}\Psi_{\bar{\eta}}(j_{!*}\cf[d])\xrightarrow{T-\Id} i_{x}^{!}R^{-1}\Psi_{\bar{\eta}}(j_{!*}\cf[d])\Big\},&\text{ for }m=-1,\\
R^{m}\Psi_{\bar{\eta}}(j_{!*}\cf[d])_{{x}},& \text{ for }-d\le m\le -2.
\end{cases}
$$
Take Poincar\'e dual, we get
$$
\ch^{n}(i_{x}^{!}(j_{!*}\cf[d]))=\begin{cases}
{\rm Coker}\Big\{i_{x}^{*}R^{-1}\Psi_{\bar{\eta}}(j_{!*}\cf[d])\xrightarrow{T-\Id} i_{x}^{!}R^{-1}\Psi_{\bar{\eta}}(j_{!*}\cf[d])\Big\},&\text{ for }n=-1,\\
i_{x}^{!}R^{n}\Psi_{\bar{\eta}}(j_{!*}\cf[d]),& \text{ for }0\le n\le d-2.
\end{cases}
$$
Moreover, we have
$$
R^{m}\Phi_{\bar{\eta}}(j_{!*}\cf[d])_{x}=\begin{cases}
{\rm Im}\Big\{i_{x}^{*}R^{-1}\Psi_{\bar{\eta}}(j_{!*}\cf[d])\xrightarrow{T-\Id} i_{x}^{!}R^{-1}\Psi_{\bar{\eta}}(j_{!*}\cf[d])\Big\},&\text{ for }m=-1,\\
0,& \text{ otherwise}.
\end{cases}
$$

\end{thm}

\begin{proof}

Let $\{X_{\alpha}\}_{\alpha\in \kaa}$ be a Whitney stratification of $X$ such that $j_{!*}\cf$ is constructible with respect to it. Working with a transversal slice to the stratum containing $x$, we can assume without loss of generality that $\{x\}$ is itself a stratum. 
By assumption, $p$ is a generic projection, we can assume that its fibers are transversal to the strata containing $0$ at the frontier. In particular, the restriction of $p$ to each such stratum is smooth. The local acyclicity of the smooth morphisms (cf. \cite{sga4.5} [Arcata]) then implies that $R\Phi_{\bar{\eta}}(j_{!*}\cf[d])$ is supported only at $x$. 
As the vanishing cycle $R\Phi_{\bar{\eta}}(j_{!*}\cf[d])[-1]$ is  perverse (cf. \cite{illusie monodromie} [Appendice]),
this implies that
$$
\ch^{m}\big(i_{x}^{*}R\Phi_{\bar{\eta}}(j_{!*}\cf[d])[-1]\big)=0\text{ for }m> 0 \quad\text{and}
\quad 
\ch^{m}\big(i_{x}^{!}R\Phi_{\bar{\eta}}(j_{!*}\cf[d])[-1]\big)=0\text{ for }m< 0.
$$
Moreover, as $R\Phi_{\bar{\eta}}(j_{!*}\cf[d])$ is supported only at $x$, we have the extra equality 
$$
i_{x}^{*}R\Phi_{\bar{\eta}}(j_{!*}\cf[d])[-1]=i_{x}^{!}R\Phi_{\bar{\eta}}(j_{!*}\cf[d])[-1].
$$ 
Hence
\begin{equation}\label{vanishing vanish}
\ch^{m}\big(i_{x}^{*}R\Phi_{\bar{\eta}}(j_{!*}\cf[d])[-1]\big)=0\text{ for }m\neq 0.
\end{equation}
Let $i_{X_{0}}:X_{0}\hookrightarrow X$ be the natural inclusion, recall that we have a distinguished triangle
\begin{equation}\label{nearby vanishing}
{i}_{X_{0}}^{*}\to R\Psi_{\bar{\eta}}\xrightarrow{\can} R\Phi_{\bar{\eta}}\xrightarrow{+1}.
\end{equation}
We deduce from it a long exact sequence
\begin{equation}\label{nearby long derived}
\cdots\to \ch^{m}(i_{x}^{*}j_{!*}\cf[d])\to R^{m}\Psi_{\bar{\eta}}(j_{!*}\cf[d])_{x}\xrightarrow{\rm can}
R^{m}\Phi_{\bar{\eta}}(j_{!*}\cf[d])_{x}\to \ch^{m+1}(i_{x}^{*}j_{!*}\cf[d])\to \cdots.
\end{equation}
The vanishing result (\ref{vanishing vanish}), together with the support condition $\ch^{m}(i_{x}^{*}j_{!*}\cf[d])=0$ for $m\ge 0$, implies that 
$$
\ch^{m}((j_{!*}\cf[d])_{x})=\begin{cases}
{\rm Ker}\big\{R^{-1}\Psi_{\bar{\eta}}(j_{!*}\cf[d])_{x}\xrightarrow{\rm can}
R^{-1}\Phi_{\bar{\eta}}(j_{!*}\cf[d])_{x}\big\},&\text{ for }m=-1,\\
R^{m}\Psi_{\bar{\eta}}(j_{!*}\cf[d])_{x},& \text{ for }-d\le m\le -2,\\
0,&\text{otherwise}.
\end{cases}
$$
Take the Poincar\'e-Verdier dual of the distinguished triangle (\ref{nearby vanishing}), we get another distinguished triangle
\begin{equation}\label{vanishing nearby}
{i}_{X_{0}}^{!}[1]\to R\Phi_{\bar{\eta}}\xrightarrow{\var} R\Psi_{\bar{\eta}}\xrightarrow{+1}.
\end{equation}
For $K=j_{!*}\cf[d]$, similar vanishing results as above imply that 
\begin{equation}\label{vanishing inj}
0\to R^{-1}\Phi_{\bar{\eta}}(j_{!*}\cf[d])_{x}\xrightarrow{\rm var}
i_{x}^{!}R^{-1}\Psi_{\bar{\eta}}(j_{!*}\cf[d])\to \ch^{1}(i_{x}^{!}j_{!*}\cf[d])\to 0.
\end{equation}
Note that we have factorization $T-\Id=\var\circ \can$. Moreover, as the vanishing cycle is supported only at $x$, the factorization can be refined to
\begin{equation}\label{monodromy factorize refined}
\begin{tikzcd}
R\Psi_{\bar{\eta}}(K)\arrow[d, "T-\Id"']\arrow[r, "\can"]& i_{x,*} i_{x}^{*}R\Phi_{\bar{\eta}}(K)\arrow[d,"\var"]\\
R\Psi_{\bar{\eta}}(K) & i_{x,{*}} i_{x}^{!}R\Psi_{\bar{\eta}}(K) \arrow[l].
\end{tikzcd}
\end{equation}
With the injectivity at the left of the equation (\ref{vanishing inj}), we get
$$
{\rm Ker}\big\{R^{-1}\Psi_{\bar{\eta}}(j_{!*}\cf[d])_{x}\xrightarrow{\rm can}
R^{-1}\Phi_{\bar{\eta}}(j_{!*}\cf[d])_{x}\big\}
=
{\rm Ker}\big\{i_{x}^{*}R^{-1}\Psi_{\bar{\eta}}(j_{!*}\cf[d])\xrightarrow{T-\Id} i_{x}^{!}R^{-1}\Psi_{\bar{\eta}}(j_{!*}\cf[d])\big\}.
$$
This finishes the proof of the first assertion. The second assertion follows by Poincar\'e duality.

For the third assertion, the vanishing result (\ref{vanishing vanish}) states that
$$
R^{m}\Phi_{\bar{\eta}}(j_{!*}\cf[d])_{x}=0,\quad \text{for }m\neq -1.
$$ 
The long exact sequence (\ref{nearby long derived}) and the property of $j_{!*}\cf[d]$ imply the surjection
$$
R^{-1}\Psi_{\bar{\eta}}(j_{!*}\cf[d])_{x}\xrightarrow{\rm can}
R^{-1}\Phi_{\bar{\eta}}(j_{!*}\cf[d])_{x}\to 0.
$$
The assertion about $R^{-1}\Phi_{\bar{\eta}}(j_{!*}\cf[d])_{x}$ then follows from it and the factorization (\ref{monodromy factorize refined}), together with the injectivity at the left of the equation (\ref{vanishing inj}).

\end{proof}

\paragraph{(2) Iterated fibrations of the nearby fiber}

We proceed to analyze the nearby cycle in theorem \ref{ic via nearby}. Topologically, there exist sufficiently small positive real numbers $\ep_{0}$ and $\delta_{0}$, such that for any $0<\ep<\ep_{0}$ and $0<\delta<\delta_{0}$ we have
\begin{align}
&i_{x}^{*}R^{m}\Psi_{\bar{\eta}}(j_{!*}\cf)=H^{m}(X\cap B_{0}(\ep)\cap p^{-1}(t), j_{!*}\cf), \quad \forall\, 0\neq t\in D_{0}(\delta), \label{nearby to nearby}
\\
&i_{x}^{!}R^{n}\Psi_{\bar{\eta}}(j_{!*}\cf)=H_{c}^{n}(X\cap B_{0}(\ep)\cap p^{-1}(t), j_{!*}\cf), \quad \forall\, 0\neq t\in D_{0}(\delta), \label{nearby to nearby c}
\end{align}  
where $B_{0}(\ep)$ is the open ball of radius $\ep$ with center $0$ in $\bc^{N}$ and $D_{0}(\delta)$ is the open disk of radius $\delta$ with center $0$ in $\bc$. 
The cohomologies at the right hand side of both (\ref{nearby to nearby}) and (\ref{nearby to nearby c}) can be calculated by iterated fiberations. 
Note that the restriction of $j_{!*}\cf$ to the nearby fiber $X\cap B_{0}(\ep)\cap p^{-1}(t)$ remains an intersection complex, as $p^{-1}(t)$ intersects transversally with the Whitney stratification $\{X_{\alpha}\}_{\alpha\in \kaa}$ inside $X\cap B_{0}(\ep)$.

As our analysis concerns the restriction of $j_{!*}\cf$ to $X\cap B_{0}(\ep)$, there is no loss of generality to assume that the cardinal of the strata $\{X_{\alpha}\}_{\alpha\in \kaa}$ is finite. 
Applying proposition \ref{teissier local polar} and \ref{hm transversal} to the Zariski closure of each stratum $X_{\alpha}$, as they are of finite cardinal, we obtain the existence of a dense open subscheme $\Omega_{\{X_{\alpha}\}_{\alpha\in \kaa}}$ of the partial flag variety parametrizing the sequences of linear subspaces
$$
\{0\}\subset  H_{d}\subsetneq \cdots \subsetneq H_{2}\subsetneq H_{1}\subsetneq \bc^{N}, \quad \dim(H_{i})=N-i \text{ for all } 1\le i\le d,
$$
for which the conclusions of both propositions hold for all $\overline{X_{\alpha}}, \alpha\in \kaa$. 
Let $\pi_{i}:\bc^{N}\to \bc^{i}$ be the projection with kernel $H_{i}$. 
Then the critical locus of the restriction of $\pi_{i}$ to the germ $(X,x)$ is the local polar variety $P_{0}(X, H_{i})$, which is either empty or reduced of pure dimension $i-1$, and the restriction of $\pi_{i}$ to it is finite. 
Let $t_{i}\in \bc^{i}$ be a point in general position and sufficiently close to $0$, then $\pi_{i+1}:\bc^{N}\to \bc^{i+1}$ induces a projection 
$$p_{i+1}: X\cap B_{0}(\ep)\cap \pi_{i}^{-1}(t_{i})\to \bc^{1},$$ 
and the critical locus of $p_{i+1}$ is the finite subscheme $\pi_{i}^{-1}(t_{i})\cap P_{0}(X, H_{i+1})$.
As $P_{0}(X, H_{i+1})$ is a proper closed subscheme of $X$, we obtain that the critical points of $p_{i+1}$ doesn't lie on the boundary $X\cap \partial B_{0}(\ep)\cap \pi_{i}^{-1}(t_{i})$.
Similar results hold for $\overline{X_{\alpha}}$ in place of $X$.

Let $p_{1}:X\to \bc^{1}$ be the restriction of $\pi_{1}$ to $X$, it is clear that theorem \ref{ic via nearby} holds for it and we get the equations (\ref{nearby to nearby}) and (\ref{nearby to nearby c}) for $p=p_{1}$ and $t=t_{1}$.
The projection $p_{2}:X\cap B_{0}(\ep)\cap \pi_{1}^{-1}(t_{1})\to \bc^{1}$ makes the first fiberation of the nearby fiber. 
It has critical points at the intersection $\pi_{1}^{-1}(t_{1})\cap P_{0}(X,H_{2})$, away from the boundary, hence both $Rp_{2,*}j_{!*}\cf$ and $Rp_{2,!}j_{!*}\cf$ are locally constant away from the images of the critical points. The fiber of $Rp_{2,*}j_{!*}\cf$ over a generic point of ${\rm Im}(p_{2})$ can be identified with $$H^{*}(X\cap B_{0}(\ep)\cap \pi_{2}^{-1}(t_{2}), j_{!*}\cf)$$ for some point $t_{2}$ in general position and sufficiently closed to $0$, 
and $H^{*}(X\cap B_{0}(\ep)\cap \pi_{1}^{-1}(t_{1}), j_{!*}\cf)$ can be built up from it and the vanishing cycles of the sheaf $j_{!*}\cf$ for the morphism $p_{2}$ at the critical points $\pi_{1}^{-1}(t_{1})\cap P_{0}(X,H_{2})$.
Similar results hold for $H_{c}^{*}(X\cap B_{0}(\ep)\cap \pi_{1}^{-1}(t_{1}), j_{!*}\cf)$.

To calculate $H^{*}(X\cap B_{0}(\ep)\cap \pi_{2}^{-1}(t_{2}), j_{!*}\cf)$ and $H_{c}^{*}(X\cap B_{0}(\ep)\cap \pi_{2}^{-1}(t_{2}), j_{!*}\cf)$, we make the second fiberation $p_{3}:X\cap B_{0}(\ep)\cap \pi_{2}^{-1}(t_{2})\to \bc^{1}$.  
Again, it has critical points at the intersection $\pi_{2}^{-1}(t_{2})\cap P_{0}(X,H_{3})$, away from the boundary, hence both $Rp_{3,*}j_{!*}\cf$ and $Rp_{3,!}j_{!*}\cf$ are locally constant away from the images of the critical points. The fiber of $Rp_{3,*}j_{!*}\cf$ over a generic point of ${\rm Im}(p_{3})$ can be identified with $H^{*}(X\cap B_{0}(\ep)\cap \pi_{3}^{-1}(t_{3}), j_{!*}\cf)$ for some point $t_{3}$ in general position and sufficiently closed to $0$, 
and $H^{*}(X\cap B_{0}(\ep)\cap \pi_{2}^{-1}(t_{2}), j_{!*}\cf)$ can be built up from it and the vanishing cycles of the sheaf $j_{!*}\cf$ for the morphism $p_{3}$ at the critical points $\pi_{2}^{-1}(t_{2})\cap P_{0}(X,H_{3})$. Similar results hold for  $H_{c}^{*}(X\cap B_{0}(\ep)\cap \pi_{2}^{-1}(t_{2}), j_{!*}\cf)$.

This process can be iterated. At the end, we need to calculate the vanishing cycle of the sheaf $j_{!*}\cf$ for the morphisms $p_{i+1}: X\cap B_{0}(\ep)\cap \pi_{i}^{-1}(t_{i})\to \bc^{1}$ at their critical points $\pi_{i}^{-1}(t_{i})\cap P_{0}(X, H_{i+1})$, and these datum can be assembled to get both $i_{x}^{*}R\Psi_{\bar{\eta}}j_{!*}\cf$ and $i_{x}^{!}R\Psi_{\bar{\eta}}j_{!*}\cf$.
This can be facilitated with the stratified Morse theory of Goresky and MacPherson \cite{gm morse}. Let $\varphi_{i+1}:X\cap B_{0}(\ep)\cap \pi_{i}^{-1}(t_{i})\to \br$ be a stratified Morse function which approximates sufficiently well the function ${\rm Re}(p_{i+1})$.  
Let $v_{i}\in \br$ be its smallest critical value, then 
$$
H^{*}\big(X\cap B_{0}(\ep)\cap \pi_{i}^{-1}(t_{i})\cap \varphi_{i+1}^{-1}(-\infty, v_{i}), j_{!*}\cf\big)=H^{*}\big(X\cap B_{0}(\ep)\cap \pi_{i+1}^{-1}(t_{i+1}), j_{!*}\cf\big),
$$ 
and $H^{*}\big(X\cap B_{0}(\ep)\cap \pi_{i}^{-1}(t_{i}), j_{!*}\cf\big)$ is built up from it and the Morse groups of $j_{!*}\cf$ for the function $\varphi_{i+1}$. 
Similar results hold for $H_{c}^{*}\big(X\cap B_{0}(\ep)\cap \pi_{i}^{-1}(t_{i}), j_{!*}\cf\big)$.
Note that as both $\pi_{i}$ and $t_{i}$ are chosen generically, $\pi_{i}^{-1}(t_{i})$ intersects transversally with the Whitney stratification $\{X_{\alpha}\}_{\alpha\in \kaa}$ inside $X\cap B_{0}(\ep)$, hence the restriction of $j_{!*}\cf$ to $X\cap B_{0}(\ep)\cap \pi_{i}^{-1}(t_{i})$ remains an intersection complex.
This makes the calculation of the Morse groups of $j_{!*}\cf$ for the function $\varphi_{i+1}$ much more easier than that of the vanishing cycles of $j_{!*}\cf$ for the projection $p_{i+1}$ (cf. \cite{gm morse}, Part II, Chap. 6).

Note that the local polar varieties $P_{0}(X,H_{i})$ plays a fundametal role in the above analysis, this may help understanding the condition $(1)$ in theorem \ref{whitney criteria}.

\bigskip
\small
\noindent
\begin{tabular}{ll}
&Zongbin {\sc Chen} \\ 
\\
&School of mathematics, Shandong University\\
&JiNan, 250100,\\
&Shandong, China \\
&email: {\tt zongbin.chen@email.sdu.edu.cn}

\end{tabular}

\end{document}